\documentclass[11pt]{article}

\usepackage[T1]{fontenc}
\usepackage{lmodern}
\usepackage{microtype}
\usepackage[margin=1in]{geometry}
\usepackage{amsmath,amssymb,amsthm}
\usepackage{enumitem}
\usepackage[authoryear,round,comma]{natbib}
\usepackage[colorlinks=true,allcolors=blue]{hyperref}

\numberwithin{equation}{section}

\theoremstyle{plain}
\newtheorem{theorem}{Theorem}[section]
\newtheorem{lemma}[theorem]{Lemma}
\newtheorem{proposition}[theorem]{Proposition}
\newtheorem{assumption}[theorem]{Assumption}
\newtheorem{corollary}[theorem]{Corollary}

\newcommand{\E}{\mathbb E}
\newcommand{\Pp}{\mathbb P}
\newcommand{\R}{\mathbb R}
\newcommand{\RR}{\mathbb{R}}
\newcommand{\cG}{{\cal G}}
\newcommand{\cQ}{{\cal Q}}
\newcommand{\cS}{{\cal S}}
\newcommand{\al}{{\alpha}}
\newcommand{\eps}{{\epsilon}}
\newcommand{\ga}{{\gamma}}
\newcommand{\sig}{\sigma}
\newcommand{\epsN}{{\epsilon_n}}
\newcommand{\etaN}{{\eta_n}}
\newcommand{\rhoR}{{\rho_R}}
\newcommand{\wh}{\widehat}
\newcommand{\1}{\mathbf 1}
\newcommand{\abs}[1]{\left|#1\right|}
\newcommand{\dd}{\,\mathrm d}
\providecommand{\ip}[2]{\left\langle #1,#2\right\rangle}
\providecommand{\e}{\mathrm e}
\providecommand{\epsn}{\varepsilon_n}
\providecommand{\cC}{\mathcal C}
\providecommand{\F}{\mathcal F}

\title{Pathwise Uniqueness for Cyclically Catalytic Branching Diffusions}
\author{%
  Jie Xiong\thanks{Department of Mathematics and Shenzhen International Center
    for Mathematics, Southern University of Science and Technology, Shenzhen, China.
    \href{mailto:xiongj@sustech.edu.cn}{\nolinkurl{xiongj@sustech.edu.cn}}.
    ORCID \href{https://orcid.org/0000-0001-5900-1284}{0000-0001-5900-1284}.}
  \and
  Wen Xu\thanks{School of Mathematical Sciences, Peking University, Beijing, China.
    \href{mailto:xuwen@math.pku.edu.cn}{\nolinkurl{xuwen@math.pku.edu.cn}}.
    ORCID \href{https://orcid.org/0009-0006-4986-3554}{0009-0006-4986-3554}.}
}
\date{}
\hypersetup{
  pdftitle={Pathwise Uniqueness for Cyclically Catalytic Branching Diffusions},
  pdfauthor={Jie Xiong and Wen Xu}
}

\begin{document}
\maketitle

\begin{abstract}
We study pathwise uniqueness for cyclic catalytic stochastic differential equations whose state-dependent square-root diffusion coefficients are non-Lipschitz and degenerate on the boundary. The approach is the direct construction of a strong solution using a Malliavin compactness criterion. 
The key is the development of a new family of boundary-sensitive weighted Malliavin
estimates for the tangent processes of the smooth approximations. Pathwise uniqueness then follows from the dual Yamada--Watanabe argument together with the weak uniqueness available in the literature. 
\end{abstract}

\begin{quote}
\small
\noindent\textbf{MSC 2020}\quad Primary 60H10; secondary 60H07, 60J60, 60J80, 60K35.
\par\noindent\textbf{Keywords}\quad Malliavin compactness; strong existence;
cyclically catalytic branching; state-dependent coefficients;
stochastic differential equation; pathwise uniqueness.
\end{quote}

\section{Introduction}

\label{sec1}

Consider a particle system initialized at time $t=0$ with $k_n$ particles at positions
$x_i^n\in\mathbb R$, $i=1,2,\ldots,k_n$. Each particle carries an exponential
clock with rate $n\lambda$, independently of all other particles, where
$\lambda>0$ is a fixed constant. When its clock rings, the particle either
splits into two offspring or dies, each with probability $1/2$. We label each
particle by a multi-index $\alpha$. For example, $\alpha=(3,1)$ represents a
second-generation particle, namely the first offspring of the third particle
in the first generation. Between branching times, each particle $\alpha$
evolves according to a Brownian motion $B^\alpha$, which is independent of all
other random processes. We then define the measure-valued process
\[
\mu_t^n=\frac1n\sum_{\alpha\sim t}\delta_{x_\alpha^n(t)},
\]
where $\alpha\sim t$ means that particle $\alpha$ is alive at time $t$.

Under suitable conditions, it was proved that $\{\mu^n_\cdot\}$ converges
weakly to $\mu_\cdot$, which is the unique solution to the following martingale
problem. Namely, $\mu_\cdot$ is an $\mathcal M_F(\mathbb R)$-valued process such
that, for every $\phi\in C_b^2(\mathbb R)$,
\[
M_t(\phi):=\langle\mu_t,\phi\rangle-\langle\mu_0,\phi\rangle
-\int_0^t\left\langle\mu_s,\frac12\phi''\right\rangle\dd s
\]
is a continuous martingale with quadratic variation
\[
\langle M(\phi)\rangle_t=\int_0^t\langle\mu_s,\lambda\phi^2\rangle\dd s.
\]
Here $\mathcal M_F(\mathbb R)$ denotes the space of finite Borel measures on
$\mathbb R$, and $\langle\mu,f\rangle$ denotes integration of $f$ with respect
to $\mu$. The limiting process is called a Dawson--Watanabe process, a
superprocess, or a super-Brownian motion (SBM). We refer the reader to Dawson
\citeyearpar{Daw} and Watanabe \citeyearpar{Wat} for the origins of these processes.

Independently, Konno and Shiga \citeyearpar{KS} and Reimers \citeyearpar{Rei} proved that the
superprocess $\mu_t$ admits a density, i.e. $\mu_t(\mathrm{d} x)=v_t(x)\mathrm{d} x$, where $v_t(x)$ is the unique weak solution to the following stochastic partial
differential equation (SPDE):
\begin{equation}\label{eq2a5}
	\partial_t v_t(x)=\frac12\Delta v_t(x)+\sqrt{\lambda v_t(x)}\,\dot B_{t,x},
	\qquad v_0=\frac{\mu_0(\dd x)}{\dd x}.
\end{equation}
Here $\dot B_{t,x}$ denotes space-time white noise.

The pathwise uniqueness problem for \eqref{eq2a5} has attracted considerable
attention. When the underlying spatial motion is collapsed to a single point,
the equation reduces to a Feller diffusion, whose pathwise uniqueness follows
from the classical Yamada--Watanabe argument. When the noise is white in time
and colored in space, this problem was solved by Mytnik, Perkins and Sturm
\citeyearpar{MPS}. When $\sqrt{v_t(x)}$ is replaced by $v_t(x)^\alpha$ with
$\alpha\ge3/4$, pathwise uniqueness was established by Mytnik and Perkins
\citeyearpar{MP}. There are also nonuniqueness results below the critical H\"older
exponent. Burdzy, Mueller and Perkins \citeyearpar{BMP} proved pathwise nonuniqueness
for nonnegative solutions when $0<\alpha<1/2$. Mueller, Mytnik and Perkins
\citeyearpar{MMP} proved nonuniqueness for signed solutions when
$1/2\le\alpha<3/4$, which includes the interesting case of $\al=1/2$, by considering the equation in
which $v_t(x)^\alpha$ is replaced by $|v_t(x)|^\alpha$. Instead of working directly with the
density-valued SPDE, Xiong \citeyearpar{Xio} considered the distribution-function-valued process
associated with the super-Brownian motion and  proved pathwise uniqueness for this alternative SPDE by an extended
Yamada--Watanabe argument.

When the branching rate depends on the concentration of another SBM population, the density process $v_t$ satisfies
\begin{equation*}
	\partial_t v_t(x)=\frac12\Delta v_t(x)+\sqrt{u_t(x)v_t(x)}\,\dot B_{t,x},
	\qquad v_0=\mu,
\end{equation*}
where $u_t(x)$ is the density of another SBM independent of the noise $\dot{B}_{tx}$. In this setting, $u_t(x)$ is called the
catalyst and $v_t(x)$ the reactant.  The resulting process
$v_t(x)$ is known as an {\em SBM with an SBM catalyst},   which was studied by Dawson and Fleischmann \citeyearpar{DF}.

A further study is the following mutually catalytic branching model: 
\begin{equation*}
	\left\{
	\begin{aligned}
		\partial_t v_t(x)&=\frac12\Delta v_t(x)
		+\lambda_1\sqrt{u_t(x)v_t(x)}\,\dot B_{t,x},\\
		\partial_t u_t(x)&=\frac12\Delta u_t(x)
		+\lambda_2\sqrt{u_t(x)v_t(x)}\,\dot W_{t,x},
	\end{aligned}
	\right.
\end{equation*}
where $\dot W_{t,x}$ is independent of $\dot B_{t,x}$. In the symmetric case
$\lambda_1=\lambda_2$, Dawson and Perkins \citeyearpar{DP} established weak uniqueness
using self-duality.

A natural extension is  the following cyclically catalytic model:
\begin{equation*}
	\partial_t u_t^k(x)=\frac12\Delta u_t^k(x)
	+\sqrt{u_t^k(x)u_t^{k+1}(x)}\,\dot W_{t,x}^k,
	\qquad k=1,2,\ldots,K,
\end{equation*}
where we use the convention $u_t^{K+1}(x)=u_t^1(x)$. Although weak uniqueness remains open for this model, Fleischmann and Xiong
\citeyearpar{FX} constructed a strong Markov solution via a Markov selection argument
and studied its long-time segregation and extinction-survival properties.

This motivates the finite-dimensional uniqueness problem. When the underlying
spatial motion is collapsed to a single point, one obtains the following $K$ dimensional affine cyclic
catalytic  stochastic
differential equations (SDE):
\begin{equation}\label{eq0425a}
	\dd X_t^k=\left(\sum_{j=1}^K q_{kj}X_t^j+\theta_k\right)\dd t
	+\sqrt{\gamma_kX_t^kX_t^{k+1}}\,\dd W_t^k,
	\qquad k=1,2,\ldots,K,
\end{equation}
where $q_{ij}\ge0$ for $i\ne j$, $\theta_k \geq 0,\gamma_k>0$, and
$W^1,\ldots,W^K$ are mutually independent standard Brownian motions. We use
the cyclic conventions $X_t^{K+1}=X_t^1$ and $q_{K,K+1}=q_{K,1}$. In
particular, if $(q_{ij})$ is the $Q$-matrix of a Markov chain, then
\eqref{eq0425a} describes a super-Markov chain. 

In this direction, Athreya et al. \citeyearpar{ABBP} studied
\begin{equation}\label{eq0425b}
	\dd X_t^k=b_k(X_t)\dd t+\sqrt{X_t^k\gamma_k(X_t)}\,\dd W_t^k,
	\qquad k=1,2,\ldots,K.
\end{equation}
They proved weak uniqueness under the assumptions that $\gamma_k(x)>0$ for all
$x$, that $b_k(x)>0$ for $x\in\partial\RR^K_+$, and
that  $b_k(x)$ have linear growth. These assumptions, however, are
not satisfied by \eqref{eq0425a}.
More general equations were studied by Bass and Perkins \citeyearpar{BP,BP2}. However, SDE (\ref{eq0425a}) is still not covered
by these extensions. Finally, Dawson and Perkins \citeyearpar{DP2} obtained a further extension by
considering the following SDE:
\begin{equation}\label{eq0425c}
	\dd X_t^k=b_k(X_t)\dd t
	+\sqrt{\gamma_k(X_t)X_t^kX_t^{k+1}}\,\dd W_t^k,
	\qquad k=1,2,\ldots,K.
\end{equation}
Under suitable H\"older continuity and boundary conditions  on the coefficients, with
$\gamma_k(x)$ strictly positive and locally bounded away from zero, they proved
weak uniqueness for the associated martingale problem.

The weak uniqueness results naturally lead to the stronger question of
pathwise uniqueness for \eqref{eq0425a} and \eqref{eq0425c}. In the case
$q_{ij}=0$ for $i\ne j$ and $\theta_k=0$, Dawson, Fleischmann and Xiong
\citeyearpar{DFX} established pathwise uniqueness for \eqref{eq0425a}. Their argument relies crucially on the
special absorbing property that any component, once it reaches zero, remains
at zero forever.  In the same spirit with $\sqrt{\ga_k X^k_tX^{k+1}_t}$ being replaced by $\sqrt{X^k_tf_k(X_t)}$, 
the pathwise
uniqueness problem was studied by He \citeyearpar{He}. The primary objective of this article is to construct a strong solution to the state-dependent equation \eqref{eq0425c} under the assumptions specified below. Combined with the corresponding weak uniqueness result, this further yields pathwise uniqueness. The affine equation \eqref{eq0425a} is then obtained as a special case. In particular, for the affine system, the strong well-posedness result remains valid under the critical condition
$
q_{i,i+1}\ge \frac{\gamma_i}{4}, i=1,\ldots,K.
$

Finally, we point out that pathwise uniqueness results for multidimensional
SDEs with non-Lipschitz coefficients remain limited. 
Apart from the works of DeBlassie
\citeyearpar{Deb} and Swart \citeyearpar{Swa2},  few results seem to be available in
this direction.  In
particular, DeBlassie \citeyearpar{Deb} posed pathwise uniqueness for equations of the
form \eqref{eq0425b} as an open problem. The results below give a positive
partial answer under explicit structural conditions. We also point out here that when $K=2$, we have established in \citet{XX} the pathwise uniqueness of \eqref{eq0425a} without any extra condition such as \eqref{eq:reserve-def} below. This article addresses arbitrary cyclic dimension, more importantly,
treats the more general state-dependent model \eqref{eq0425c}.

The initial condition is an interior point of the nonnegative orthant:
\begin{equation}\label{eq:interior-initial}
	X_0=x_0=(x_0^1,\ldots,x_0^K)\in(0,\infty)^K.
\end{equation}
A solution of  \eqref{eq0425c}, when it exists, will be denoted by
$X_t=(X^1_t,\ldots,X^K_t)$.  We refer to Ikeda and Watanabe \citeyearpar{iw} and Cherny
\citeyearpar{Che} for the relevant notions of weak, strong, and pathwise solutions.

The key step in the proof is to apply the global Malliavin compactness
argument. Related Malliavin-compactness arguments were used by Meyer-Brandis and
Proske \citeyearpar{MeyerProske} for SDEs with bounded measurable drift and
additive nondegenerate noise, and by Shaposhnikov \citeyearpar{Shaposhnikov}
for a special two-dimensional degenerate system with additive noise. In contrast, our diffusion coefficient is state dependent, of square-root
type, and degenerates at the boundary. 

To this end, we impose conditions ensuring that,
when an individual coordinate approaches zero, the corresponding drift is
sufficiently strong relative to the square-root diffusion coefficient.
Therefore, we impose the following conditions.

\begin{assumption}\label{ass:regularity}
	For each $i=1,\ldots,K$:
	\begin{enumerate}[label=\textnormal{(\roman*)},leftmargin=3em]
		\item $b_i$ and $\gamma_i$ admit $C^2$ extensions to an open neighborhood of
		$\mathbb R_+^K$;
		\item there are constants
		$0<\underline\gamma\le\overline\gamma<\infty$ such that
		\begin{equation*}
			\underline\gamma\le\gamma_i(x)\le\overline\gamma,
			\qquad x\in\mathbb R_+^K;
		\end{equation*}
		\item all derivatives of $b_i$ and $\gamma_i$ up to order two are bounded
		on compact subsets of $\mathbb R_+^K$;
		\item the drift has linear growth:
		\begin{equation*}
			|b(x)|\le C_b(1+|x|),
			\qquad x\in\mathbb R_+^K.
		\end{equation*}
	\end{enumerate}
\end{assumption}

\begin{assumption}\label{ass:structure}
	For every $i=1,\ldots,K$, there exist functions
	$\beta_i,m_i:\mathbb R_+^K\to\mathbb R$ admitting $C^2$ extensions to an open
	neighborhood of $\mathbb R_+^K$,  such that
	\begin{equation}\label{eq:drift-decomp}
		b_i(x)=x_i\beta_i(x)+m_i(x),
		\quad m_i(x)\ge0,
		\quad x\in\mathbb R_+^K,
	\end{equation}
	and
	\begin{equation}\label{eq:reserve-def}
		r_i(x):=
		m_i(x)-\frac{1}{4}\gamma_i(x)x_{i+1}\ge0,
		\quad x\in\mathbb R_+^K.
	\end{equation}
	Moreover, for every $R>0$, there exists $C_R<\infty$ such that
	\begin{equation}\label{eq:weighted-transverse}
		x_j|\partial_jr_i(x)|^2\le C_R r_i(x)
	\end{equation}
	for all $i=1,\ldots,K$, all $j\ne i$, and all $x\in[0,R]^K$.
\end{assumption}

The following is the main result on the strong existence of solutions to \eqref{eq0425c}.

\begin{theorem}[Strong existence for  \eqref{eq0425c}]\label{thm:main}
	Suppose that  \eqref{eq:interior-initial} holds and Assumptions \ref{ass:regularity} and \ref{ass:structure} are satisfied. Then \eqref{eq0425c} admits a
	continuous nonnegative strong solution adapted to the completed natural
	filtration of the prescribed Brownian motion $W$. Moreover, for every $T>0$,
	\begin{equation}\label{eq:global-second-moment}
		\mathbb E\sup_{0\le t\le T}|X_t|^2\le C_T(1+|x_0|^2).
	\end{equation}
\end{theorem}

By Theorem~\ref{thm:main}, a strong solution to \eqref{eq0425c} exists and hence, in particular, a weak solution exists. Dawson and Perkins \citeyearpar[Lemma~5]{DP2} show that every weak solution remains in the natural state space
\[
S=\left\{x\in\mathbb R_+^K:\prod_{i=1}^K(x_i+x_{i+1})>0\right\}.
\]
Thus, if $x\in S$ and $x_i=0$, then necessarily $x_{i+1}>0$. By Assumption~\ref{ass:structure},
\[
b_i(x)=m_i(x)\ge \frac{1}{4}\gamma_i(x)x_{i+1}>0.
\]
Together with the  Assumption~\ref{ass:regularity}, Dawson and Perkins \citeyearpar[Theorem~4]{DP2} gives the weak uniqueness for \eqref{eq0425c}.

It follows from the dual Yamada-Watanabe argument (cf. Theorem 3.2 in \citet{Che}) that the existence of the strong solution, together with the weak uniqueness (cf. Theorem 4 and Lemma 5 in \citet{DP2}), 
implies the pathwise uniqueness.

\begin{corollary}[Pathwise uniqueness for  \eqref{eq0425c}]\label{cor:pathwise-general}
	Under the hypotheses of Theorem \ref{thm:main},  the pathwise uniqueness of the solution to the SDE
	\eqref{eq0425c} holds. 
\end{corollary}

The main contribution of this paper is the use of a Malliavin compactness approach to construct strong solutions for degenerate cyclic catalytic SDEs with state-dependent, non-Lipschitz diffusion coefficients. However, although the general Malliavin compactness framework is well established, its implementation in the present degenerate setting is far from standard: in particular, verifying the second condition of the compactness criterion requires delicate control of the singular behavior near the boundary.  We overcome this difficulty by means of a series of estimates established in Section \ref{sec5}. These estimates
provide both the uniform Malliavin $L^2$ bounds and the fractional
regularity in the Malliavin parameter required for compactness.

The key innovation, and a
major contribution in its own right,
is the
development of a new family of boundary-sensitive weighted Malliavin
estimates for the tangent processes of the smooth approximations.  Instead
of estimating the Malliavin derivatives in the standard unweighted
$L^2$-norm, we introduce the regularized weighted energy \eqref{eq:Q-def}. The
weight reveals a precise cancellation between the inverse-square-root
singularities generated by the derivatives of the diffusion coefficients
and the incoming catalytic drift.  Under the structural condition
\eqref{eq:reserve-def}, 
this cancellation yields estimates that are uniform in the square-root
regularization parameter. 
This approach provides the key mechanism that makes the construction of strong solutions possible despite the boundary degeneracy and may also be applicable to other multidimensional stochastic systems with square-root-type singularities.

The rest of this article is organized as follows. Section \ref{sec2}
constructs the smooth localized approximations and establishes their positivity
and basic estimates. Section \ref{sec3} introduces conditions
MC1--MC3 and proves that these conditions imply the existence of a strong
solution. Section \ref{sec4} proves the weighted stability estimate for the
Malliavin tangent equation. Section \ref{sec5} establishes the boundary occupation and
short-time estimates needed to control the singular terms. Section \ref{sec6}
uses these estimates to prove the fractional Malliavin regularity and hence
verify MC2. Section \ref{sec7} verifies MC3 and completes the proof of the
main strong-existence theorem. The appendix contains the proof of an algebraic estimate.

Throughout the article, $C$ denotes a finite constant whose value may change
from line to line. When its dependence is relevant, we indicate it by
subscripts, for example $C_{n,R}$.

\section{Smooth boundary-linearized approximation}
\label{sec2}

As a preparation for the proof of the main result, we construct a sequence of SDEs to approximate the original one.
For $n\ge1$, $ u\ge0$, let
\begin{equation*}
	\epsN=2^{-n},\quad \etaN=\frac{\epsN}{8},\quad
	g_n(u)=\frac{u}{\sqrt{u+\epsN}},\quad
	p_n(u)=g_n(u)^2=\frac{u^2}{u+\epsN}.
\end{equation*}
Extend $g_n$ smoothly to $\R$, the extension is irrelevant because the solutions below remain non-negative.
Throughout the paper, whenever no confusion can arise, we suppress the index $n$ from the notation. Thus, for a fixed $n$, we abbreviate
\[
\eps=\epsN,\quad \eta=\etaN,\quad g=g_n,\quad p=p_n.
\]

For $x\in\mathbb R_+^K$, we define 
\begin{equation*}
	\chi_i(x):=\sqrt{\gamma_i(x)},
	\qquad
	\sigma_i(x):=\chi_i(x)\sqrt{x_ix_{i+1}},
	\qquad 
	\sigma_i^n(x)=\chi_i(x)g(x_i)g(x_{i+1}).
\end{equation*}
For $R>0$, choose $\rhoR\in C_c^1(\R^K;[0,1])$ such that
\begin{equation*}
	\rhoR(x)=1\quad\text{for }\abs{x}\le R,
	\quad
	\rhoR(x)=0\quad\text{for }\abs{x}\ge R+1,
\end{equation*}
and
\begin{equation}\label{eq:cutoff-gradient}
	\abs{\nabla\rhoR(x)}^2\le C_R\rhoR(x)(1-\rhoR(x)), \quad x\in\R^K. 
\end{equation}

In fact, we can define $\rho_R$ as follows:
\[\rho_R(x)=\left\{\begin{array}{ll}
	1&\mbox{ if } |x|\le R,\\
	(|x|-R-1)^2(|x|-R+1)^2&\mbox{ if } R<|x|<R+1,\\
	0&\mbox{ if } |x|\ge R+1.
\end{array}\right.\]
Fix $R>0$. By Assumption \ref{ass:regularity}, we use $C^2$ extensions of
$b_i$ and $\chi_i$ on a neighborhood of the support of $\rho_R$, and keep
the same notation.

Consider the following localized equation 
\begin{equation}\label{eq:localized-approx}
	\begin{cases}
		\dd X_t^{n,R,i}=\rhoR(X_t^{n,R})b_i(X_t^{n,R})\dd t
		+\rhoR(X_t^{n,R})\sigma_i^n(X_t^{n,R})\dd W_t^i,\\[1mm]
		X_0^{n,R}=x_0 \in(0,\infty)^K.
	\end{cases}
\end{equation}
For each fixed $n$, the coefficients are smooth with at most linear drift and bounded
localized diffusion.  Hence \eqref{eq:localized-approx} has a unique strong solution 
$X^{n,R}_t$ which is Malliavin differentiable (cf. Meyer-Brandis and Proske \citeyearpar{MeyerProske}). 
Denote by
$D^\ell_r X^{n,R}_t$ as the Malliavin 
derivative of $X^{n,R}_t$ with respect to the Brownian motion $W^\ell$.

For notational simplicity, we write 
\begin{equation*}
	\rho_t:=\rhoR(X_t^{n,R}),\quad
	x_i(t):=X_t^{n,R,i},\quad
	d_i(t):=x_i(t)+\eta, \quad \beta_i(t) := \beta_i(X_t^{n,R}), 
\end{equation*}
\begin{equation*}
	\chi_i(t):=\chi_i(X_t^{n,R}),\quad
	b_i(t):=b_i(X_t^{n,R}),\quad
	m_i(t):=m_i(X_t^{n,R}),\quad
	r_i(t):=r_i(X_t^{n,R}).
\end{equation*}

\begin{lemma}[Regularized square-root estimates]\label{lem:regularization}
	For every $u,v\ge0$, the following assertions hold.
	\begin{align}
		&0\le p(u)\le u,\qquad 0\le p'(u)\le1,\label{eq:p-basic}\\
		&\abs{g'(u)}^2\le\frac{16}{u+\eta},\label{eq:gprime-bound}\\
		&\alpha_\eta(u):=
		\frac{\abs{(u+\eta)g'(u)-g(u)}^2}{u+\eta}\le\frac14,\label{eq:alpha-bound}\\
		&\abs{g(u)-g(v)}^2\le\abs{p(u)-p(v)}\le\abs{u-v},\label{eq:g-holder}\\
		&g'(u)+2p(u)g''(u)\ge0,\label{eq:g-convex-combination}\\
		&\frac{g'(u)}{g(u)}\le\frac1u,
		\qquad
		\frac12\le\frac{ug'(u)}{g(u)}\le1,
		\qquad u>0.\label{eq:log-derivative}
	\end{align}
\end{lemma}

\begin{proof}
	All the assertions follow from elementary calculations. We only spell out the verification of \eqref{eq:alpha-bound}. Let  $u=\eps a$ and $\eta=\eps\delta$ with $\delta=1/8$. Then
	$$
	\alpha_\eta(u)=
	\frac{(a^2-\delta a-2\delta)^2}
	{4(a+\delta)(a+1)^3}.
	$$
	Hence $\alpha_\eta(u)\le1/4$ is equivalent to
	$$
	(a+\delta)(a+1)^3-(a^2-\delta a-2\delta)^2\ge0.
	$$
	
	The left-hand side equals
	$$
	(3+3\delta)a^3+(3+7\delta-\delta^2)a^2
	+(1+3\delta-4\delta^2)a+\delta(1-4\delta),
	$$
	which is nonnegative for $a\ge0$ and $0<\delta\le1/8$.
\end{proof}

\begin{lemma}[Strict positivity and time increments]\label{lem:positivity-increments}
	For every fixed $n,R$,
	\begin{equation}\label{eq:strict-positive}
		\Pp\bigl(X_t^{n,R,i}>0\text{ for all} \; \;  t\ge0,
		\ \text{all } i \bigr)=1.
	\end{equation}
	For every $T>0$ there is $C_{T,R}$, independent of $n$, such that
	\begin{equation}\label{eq:time-increments}
		\sup_{n\ge1}\E\abs{X_t^{n,R}-X_s^{n,R}}^2
		\le C_{T,R}\abs{t-s},
		\qquad 0\le s<t\le T.
	\end{equation}
	Moreover,
	\begin{equation}\label{eq:approx-sup-bound}
		\sup_{n\ge1}\E\sup_{0\le t\le T}\abs{X_t^{n,R}}^2
		\le (R+1\vee\abs{x_0})^2.
	\end{equation}
\end{lemma}

\begin{proof}
	Let
	\[
	\tau:=\inf\left\{t\ge0:\min_{1\le i\le K}X_t^{n,R,i}\le0\right\}.
	\]
	On $[0,\tau]$ all coordinates are nonnegative, and hence
	Assumption \ref{ass:structure} gives
	$
	b_i=x_i\beta_i+m_i.
	$  Now, let
	\[
	\sigma_i^n(X_t^{n,R})
	=x_i(t)c_i^n(t),
	\qquad
	c_i^n(t):=\chi_i(t)\frac{g(x_{i+1}(t))}{\sqrt{x_i(t)+\eps}}.
	\]
	Consequently, up to stopping time $\tau$, the $i$th coordinate satisfies the linear equation:
	\[
	\dd x_i(t)=\rho_t\beta_i(X_t^{n,R})x_i(t)\dd t
	+\rho_t c_i^n(t)x_i(t)\dd W_t^i
	+\rho_t m_i(t)\dd t.
	\]
	Let
	\[
	\mathcal{E}_i(t)=
	\exp\left\{
	\int_0^t\left(\rho_s\beta_i(X_s^{n,R})
	-\frac12\rho_s^2\abs{c_i^n(s)}^2\right)\dd s
	+\int_0^t\rho_sc_i^n(s)\dd W_s^i
	\right\}.
	\]
	For fixed $n,R$, the coefficients in this stochastic exponential are bounded on every stopped interval $[0,\tau\wedge T]$. The variation of constants  gives, for $t\le\tau\wedge T$,
	\[
	x_i(t)=\mathcal{E}_i(t)
	\left(x_0^i+\int_0^t\mathcal{E}_i(s)^{-1}\rho_sm_i(s)\dd s\right).
	\]
	The right-hand side is strictly positive because $\mathcal{E}_i>0$, $x_0^i>0$, and $m_i\ge0$.  If $\tau\le T$, continuity and the same formula at $t=\tau$ would give $x_i(\tau)>0$ for every $i$, contradicting the definition of $\tau$.  Hence $\tau>T$ almost surely and \eqref{eq:strict-positive} follows. 
	
	The localized coefficients vanish outside $\{\abs{x}<R+1\}$. Therefore,
	$\abs{X_t^{n,R}}\le R+1\vee\abs{x_0}$, which  proving \eqref{eq:approx-sup-bound}.  On this bounded state space,  we have 
	\[
	\abs{\rhoR(x)b(x)}\le C_R,
	\qquad
	\sum_{i=1}^K\abs{\rhoR(x)\sigma_i^n(x)}^2\le C_R,
	\]
	uniformly in $n$. Then, 
	Cauchy--Schwarz and It\^o isometry applied to \eqref{eq:localized-approx} yield \eqref{eq:time-increments}.
\end{proof}

\section{Main results}\label{sec3}

In this section, we introduce three conditions, MC1--MC3, and show that they imply convergence of a subsequence of the approximating solutions to a strong solution driven by the prescribed Brownian motion. Their verification for the approximation scheme of Section \ref{sec2} is deferred to the following sections.

Let $D_r^\ell X_t^{n,R,i}$ be the Malliavin derivative of the $i$th coordinate of $X_t^{n,R}$ with respect to $W^\ell$ at time $r$, with the convention that $D_r^\ell X_t^{n,R,i}=0$ for $r>t$. We write
\begin{equation*}
	\abs{D_rX_t^{n,R}}^2
	:=\sum_{i,\ell=1}^K\abs{D_r^\ell X_t^{n,R,i}}^2.
\end{equation*}

The following conditions will be assumed in the next theorem.

Conditions MC1--MC3: 
Fix $T,R>0$.
\begin{enumerate}[label=\textup{MC\arabic*.},leftmargin=2.3cm]
	\item There exist constants $C_{T,R}$ and
	$\theta>0$ such that 
	\begin{equation}\label{eq:MC1a}
		\sup_n\E\sup_{0\le t\le T}\abs{X_t^{n,R}}^2\le C_{T,R},
	\end{equation}
	and 
	\begin{equation}\label{eq:MC1b}
		\sup_n\E\abs{X_t^{n,R}-X_s^{n,R}}^2
		\le C_{T,R}\abs{t-s}^{\theta}, \qquad 0\leq s<t\leq T.
	\end{equation}
	\item For each fixed $t\in[0,T]$, there is an $\alpha>0$ such that (i)
	\begin{equation}\label{eq:MC2a}
		\sup_n\E\int_0^T\abs{D_rX_t^{n,R}}^2\dd r<\infty
	\end{equation}
	and (ii)
	\begin{equation}\label{eq:MC2b}
		\sup_n\E\int_0^T\int_0^T
		\frac{\abs{D_rX_t^{n,R}-D_{r'}X_t^{n,R}}^2}
		{\abs{r-r'}^{1+\alpha}}\dd r\dd r'<\infty.
	\end{equation}
	\item If a subsequence $X^{n_m,R}$ converges to $X^R$ in
	$L^2([0,T]\times\Omega;\R^K)$, then for each  $i$,
	\begin{align}
		&\int_0^T\E\abs{b_i(X_s^{n_m,R})-b_i(X_s^R)}\dd s\longrightarrow0,
		\label{eq:MC3a}
	\end{align}
	and 
	\begin{align}
		&\int_0^T\E\abs{\sigma_i^{n_m}(X_s^{n_m,R})
			-\sigma_i(X_s^R)}^2\dd s\longrightarrow0.
		\label{eq:MC3b}
	\end{align}
\end{enumerate}

\begin{theorem}[Identification of the approximation limit]\label{thm:MC-implies}
	For every $T,R>0$, assume that the approximating solutions \eqref{eq:localized-approx} satisfy MC1--MC3. Then, the approximating sequence admits a convergent subsequence, and its limit is a continuous nonnegative strong solution of \eqref{eq0425c}, which is  adapted to the completed natural filtration of $W$.
\end{theorem}

\begin{proof}
	We divide the proof into several steps.
	
		\medskip
	\noindent\textbf{Step 1. Compactness at rational times.}
	For each fixed $t\in[0,T]$ and $R>0$, consider the random variables $X_t^{n,R}$.
	By MC1,
	\[
	\sup_n
	\mathbb E|X_t^{n,R}|^2
	\leq
	C_R .
	\]
	By MC2, for each fixed $t\in[0,T]$ and $R>0$, the family $\left\{
	X_t^{n,R}:\ n\geq1
	\right\}$
	has uniformly bounded Malliavin derivatives and uniformly bounded
	fractional regularity in the Malliavin parameter. Therefore, by the
	Da Prato--Malliavin--Nualart compactness criterion (cf. \citealp{DPMN}, see also Corollary 1 in Shaposhnikov \citeyearpar{Shaposhnikov}), this family is
	relatively compact in $L^2(\Omega;\mathbb R^K)$.
	
	Let $\mathbb{Q}_T=\mathbb{Q}\cap[0,T]$.  A diagonal extraction gives a subsequence, still denoted $n_m$, and random variables $X_t^R$, $t\in\mathbb{Q}_T$, such that
	\begin{equation}\label{eq:rational-L2}
		X_t^{n_m,R}\longrightarrow X_t^R
		\quad\text{in }L^2(\Omega;\R^K),
		\qquad t\in\mathbb{Q}_T.
	\end{equation}
	
	\medskip
	\noindent\textbf{Step 2. Extension to all times.}
	Passing to the limit in MC1 along rational times yields
	\[
	\E\abs{X_t^R-X_s^R}^2\le C_{T,R}\abs{t-s}^{\theta},
	\qquad s,t\in\mathbb{Q}_T.
	\]
	If $t_k\in\mathbb{Q}_T$ and $t_k\to t$, then $(X_{t_k}^R)$ is Cauchy in $L^2(\Omega)$, define $X_t^R$ to be its limit.  The definition is independent of the rational sequence and
	\begin{equation}\label{eq:limit-time-regularity}
		\E\abs{X_t^R-X_s^R}^2\le C_{T,R}\abs{t-s}^{\theta},
		\qquad s,t\in[0,T].
	\end{equation}
	We claim that
	\begin{equation}\label{eq:L2-time-space-conv}
		X^{n_m,R}\longrightarrow X^R
		\quad\text{in }L^2([0,T]\times\Omega;\R^K).
	\end{equation}
	Choose a rational partition $0=t_0<\cdots<t_N=T$ with mesh at most $\delta$, and let $\pi(s)=t_k$ on $[t_k,t_{k+1})$.  Then
	\begin{align*}
		\int_0^T\E\abs{X_s^{n_m,R}-X_s^R}^2\dd s
		&\le3\int_0^T\E\abs{X_s^{n_m,R}-X_{\pi(s)}^{n_m,R}}^2\dd s\\
		&\quad+3\int_0^T\E\abs{X_{\pi(s)}^{n_m,R}-X_{\pi(s)}^R}^2\dd s\\
		&\quad+3\int_0^T\E\abs{X_{\pi(s)}^R-X_s^R}^2\dd s.
	\end{align*}
	The first and third terms are bounded by $C_{T,R}\delta^\theta$.  For a fixed partition, the middle term tends to zero by \eqref{eq:rational-L2}.  Letting first $m\to\infty$ and then $\delta\downarrow0$ proves \eqref{eq:L2-time-space-conv}.
	The subspace of progressively measurable processes is closed in $L^2([0,T]\times\Omega)$, so $X^R$ is progressively measurable with respect to the natural filtration of $W$.  	Furthermore, since $X^{n_m,R}\geq0$, the $L^2$-limit $X^R$ is
	non-negative a.e. on $[0,T]\times\Omega$, after modifying it on a null
	set.
	
	\medskip
	\noindent\textbf{Step 3.  Passage to the localized limiting equation.}
	The integral form of \eqref{eq:localized-approx} is
	\begin{equation}\label{eq:approx-integral}
		X_t^{n_m,R,i}=x_0^i+
		\int_0^t\rhoR(X_s^{n_m,R})b_i(X_s^{n_m,R})\dd s
		+\int_0^t\rhoR(X_s^{n_m,R})\sigma_i^{n_m}(X_s^{n_m,R})\dd W_s^i.
	\end{equation}
	By \eqref{eq:L2-time-space-conv}, the Lipschitz continuity of $\rhoR$, and MC3, it is easy to obtain  
	\begin{align}
		&\int_0^T\E\abs{\rhoR(X_s^{n_m,R})b_i(X_s^{n_m,R})
			-\rhoR(X_s^R)b_i(X_s^R)}\dd s\to0,
		\label{eq:localized-drift-conv}\\
		&\int_0^T\E\abs{\rhoR(X_s^{n_m,R})\sigma_i^{n_m}(X_s^{n_m,R})
			-\rhoR(X_s^R)\sigma_i(X_s^R)}^2\dd s\to0.
		\label{eq:localized-diff-conv}
	\end{align}
	The drift integrands are uniformly bounded on the localized state space.  Hence \eqref{eq:localized-drift-conv} also implies convergence in $L^2([0,T]\times\Omega)$.  Define the continuous adapted process:
	\begin{equation}\label{eq:localized-limit-RHS}
		\widetilde X_t^{R,i}:=x_0^i+
		\int_0^t\rhoR(X_s^R)b_i(X_s^R)\dd s
		+\int_0^t\rhoR(X_s^R)\sigma_i(X_s^R)\dd W_s^i.
	\end{equation}
	Then, applying the Cauchy--Schwarz and the Burkholder--Davis--Gundy inequality  to \eqref{eq:approx-integral}, \eqref{eq:localized-drift-conv}, and \eqref{eq:localized-diff-conv}, we obtain 
	\begin{equation}\label{eq:C-prob-conv}
		\E\sup_{0\le t\le T}\abs{X_t^{n_m,R}-\widetilde X_t^R}^2\longrightarrow0.
	\end{equation}
	On the other hand, \eqref{eq:L2-time-space-conv} says that the same sequence converges to $X^R$ in $L^2([0,T]\times\Omega)$.  Therefore $X^R=\widetilde X^R$ for $\dd t\otimes\dd\Pp$-almost every $(t,\omega)$.  Replacing $X^R$ by the continuous version $\widetilde X^R$, we obtain
	\begin{equation}\label{eq:localized-limit-equation}
		X_t^{R,i}=x_0^i+
		\int_0^t\rhoR(X_s^R)b_i(X_s^R)\dd s
		+\int_0^t\rhoR(X_s^R)\sigma_i(X_s^R)\dd W_s^i.
	\end{equation}
	Thus $X^R$ is a nonnegative strong solution of the localized limiting equation.
	
	\medskip
	\noindent\textbf{Step 4. A diagonal subsequence independent of $R$.}
	\medskip
	
	We now let $R\uparrow\infty$. Although the approximating processes $X^{n,R}$ depend on $R$, the subsequences used in Steps 1--3 can be chosen with a common index set for all integer $R\geq1$. Indeed, choose a subsequence for $R=1$, then from it a further subsequence for $R=2$, and continue inductively. A diagonal argument yields a single sequence of indices $(n_m)$ such that, for every integer $R\geq1$,
	$$
	\mathbb{E}\sup_{0\leq t\leq T}
	\left|X_t^{n_m,R}-X_t^R\right|^2
	\longrightarrow0.
	$$

	We next verify that the limiting processes are consistent. Fix integers $1\le R<S$ . For every $n$, the equations for $X^{n,R}$ and $X^{n,S}$ have the same coefficients as long as the common path stays in $\{\abs{x}<R\}$. Hence, by pathwise uniqueness for the smooth approximating equations,
	\begin{equation}\label{eq:approx-consistency}
		X_t^{n,R}=X_t^{n,S},
		\qquad 0\le t\le\tau_R^n,
		\qquad
		\tau_R^n:=\inf\{t\ge0:\abs{X_t^{n,R}}\ge R\}.
	\end{equation}
	By \eqref{eq:C-prob-conv}, after passing, if necessary, to a further subsequence, both $X^{n_m,R}\to X^R$ and  $X^{n_m,S}\to X^S$ uniformly on $[0,T]$ almost surely. 
	
	Define
	$$
	\tau_R^R
	:=
	\inf\{t\ge0:\abs{X_t^R}\ge R\}.
	$$
	Fix $t<\tau_R^R\wedge T$. By continuity of $X^R$, we have 
	$
	\sup_{0\le s\le t}\abs{X_s^R}<R.
	$
	Therefore, by the uniform convergence, for all sufficiently large $m$,
	$$
	\sup_{0\le s\le t}\abs{X_s^{n_m,R}}<R.
	$$
	Thus $t<\tau_R^{n_m}$, and hence \eqref{eq:approx-consistency} gives
	$$
	X_s^{n_m,R}=X_s^{n_m,S},
	\qquad
	0\le s\le t.
	$$
	Letting $m\to\infty$, we obtain
	$$
	X_s^R=X_s^S,
	\qquad
	0\le s\le t.
	$$
	Since $t<\tau_R^R\wedge T$ was arbitrary, continuity yields
	\begin{equation}\label{eq:limit-consistency}
		X_t^R=X_t^S,
		\qquad
		0\le t\le \tau_R^R\wedge T.
	\end{equation}
	\medskip
	\noindent\textbf{Step 5. Removing the localization.}
	The coefficients in \eqref{eq:localized-limit-equation} satisfy, uniformly in $R$,
	\[
	\abs{\rhoR(x)b(x)}\le C(1+\abs{x}),
	\qquad
	\sum_i\abs{\rhoR(x)\sigma_i(x)}^2
	\le C\abs{x}^2,
	\qquad x\in\R_+^K.
	\]
	The standard Burkholder--Davis--Gundy and Gronwall argument yields
	\begin{equation}\label{eq:R-uniform-moment}
		\E\sup_{0\le t\le T}\abs{X_t^R}^2
		\le C_T(1+\abs{x_0}^2),
	\end{equation}
	with $C_T$ independent of $R$.  Hence, as  $R \to \infty$, we have 
	\begin{equation}\label{eq0720a}
		\Pp(\tau_R^R\le T)
		\le\frac{C_T(1+\abs{x_0}^2)}{R^2}\longrightarrow0.
	\end{equation}
	Define
	\[
	X_t:=X_t^R
	\quad
	\text{whenever } t\leq\tau^R_R.
	\]
	The consistency above shows that this definition is independent of the
	choice of $R$.
	
	By \eqref{eq0720a},
	the process $X$ is defined on every finite time interval with
	probability one. It is continuous, non-negative, and adapted to the
	natural filtration of $W$. Moreover, for every fixed $T>0$, choosing
	$R$ sufficiently large and using the equation before $\tau^R_R$, we
	obtain, for all $t\leq T$,
	\begin{equation}\label{xu1}
		X_t^i
		=
		x^i_0
		+
		\int_0^t b_i(X_s)\,\dd s
		+
		\int_0^t
		\sigma_i(X_s) 
		\,\dd W_s^i.
	\end{equation}
	Thus,  $X$ is a non-negative strong solution of
	\eqref{eq0425c}  driven by the prescribed Brownian motion
	$W$.
	For the moment estimate,  
	applying the  BDG and Gronwall  directly to \ref{xu1} gives \eqref{eq:global-second-moment}.  This completes the proof.
\end{proof}

By Lemma \ref{lem:positivity-increments},  the family of solutions to the smooth localized equation \eqref{eq:localized-approx} satisfies condition MC1. The verification of conditions MC2 and MC3 will be carried out in the subsequent sections.
\section{The weighted quotient estimate}\label{sec4}

This section develops the weighted stability estimate needed for the verification of MC2(ii). We first states the main result, Proposition \ref{prop:weighted-stability}, which provides the required uniform weighted control of the Malliavin derivative and its increments. Subsections \ref{sub1}–\ref{sub3} are then devoted to the proof of this proposition.

Fix $\ell\in\{1,\ldots,K\}$ and $0\le r\le s$.  Write
\begin{equation}\label{eq:Y-def}
	Y_i^\ell(s;r):=D_r^\ell X_s^{n,R,i}.
\end{equation}
For a vector $u\in\R^K$, define
\begin{equation}\label{eq:R-cutoff-linear}
	R_s(u):=\ip{\nabla\rhoR(X_s^{n,R})}{u}.
\end{equation}

Taking the Malliavin derivative in \eqref{eq:localized-approx} at a deterministic time  yields, for $s\ge r$,
\begin{align}
	\dd Y_i^\ell(s;r)
	&=\left[
	\rho_s\sum_{j=1}^K\partial_jb_i(X_s^{n,R})Y_j^\ell(s;r)
	+b_i(s)R_s(Y^\ell(s;r))
	\right]\dd s \notag\\
	&\quad+\left[
	\rho_s\sum_{j=1}^K\partial_j\sigma_i^n(X_s^{n,R})Y_j^\ell(s;r)
	+\sigma_i^n(X_s^{n,R})R_s(Y^\ell(s;r))
	\right]\dd W_s^i,
	\label{eq:tangent}
\end{align}
with initial value
\begin{equation}\label{eq:tangent-seed}
	Y_i^\ell(r;r)=\delta_{i\ell}\rho_r\sigma_i^n(X_r^{n,R}).
\end{equation}

For $0<h<t$ and $0\le r\le t-h$, define, for $s\ge r+h$,
\begin{equation}\label{Z-def}
	Z_i^{\ell,r,h}(s):=Y_i^{\ell}(s;r+h)-Y_i^{\ell}(s;r).
\end{equation}
Then $Z^{\ell,r,h}$ solves the same homogeneous localized linearized equation \eqref{eq:tangent} on $[r+h,t]$ with the initial condition
\begin{equation*}\label{Z-initial}
	Z_i^{\ell,r,h}(r+h)
	=\delta_{i\ell}\rho_{r+h}\sig_i^n(X^{n,R}_{r+h})-Y_i^{\ell}(r+h;r).
\end{equation*}
We now write the equations for $Y_i^{\ell}(\cdot;r)$ and $Z_i^{\ell,r,h}$  into the following 
unified form: for 
$s\in[a,t]$,  
\begin{align}
	\dd U_i(s)
	&=\left[
	\rho_s\sum_{j=1}^K\partial_jb_i(X_s^{n,R})U_j(s)
	+b_i(s)R_s(U(s))
	\right]\dd s\notag\\
	&\quad+\left[
	\rho_s\sum_{j=1}^K\partial_j\sigma_i^n(X_s^{n,R})U_j(s)
	+\sigma_i^n(X_s^{n,R})R_s(U(s))
	\right]\dd W_s^i,
	\quad U(a)=\xi.
	\label{eq:generic-tangent}
\end{align}
Define the boundary-sensitive energy
\begin{equation}\label{eq:Q-def}
	\cQ_s(u):=\sum_{i=1}^K\frac{u_i^2}{d_i(s)}.
\end{equation}
The central estimate is the following.

\begin{proposition}[Weighted stability of the homogeneous tangent equation]\label{prop:weighted-stability}
	Fix $0\le a\le t \le T$, and let $U=(U_s)_{s\in[a,t]}$ denote the solution to the  homogeneous linearized equation \eqref{eq:generic-tangent} with the initial condition $U_a=\xi$. Then
	\begin{equation}\label{eq:weighted-stability}
		\E\cQ_t(U(t))\le C_{T,R}\E\cQ_a(\xi).
	\end{equation}
	Consequently,
	\begin{equation}\label{eq:unweighted-stability}
		\E\abs{U(t)}^2\le C_{T,R}\E\cQ_a(\xi).
	\end{equation}
\end{proposition}

The proof of this proposition is deferred to the subsections below.  We first state two consequences obtained by applying the proposition with $U=Y^{\ell}$ and $U=Z^{\ell,r,h}$, respectively.

\begin{corollary}\label{coll1}
	There is a constant $C_{T,R}$ such that
	\begin{equation}\label{eq0719a}
		\E\abs{Y^{\ell}(t;r)}^2
		\le C_{T,R}
		\E\sum_{i=1}^K
		\frac{\abs{\delta_{i\ell}\rho_r\sqrt{\gamma_i(x)}\,g(x_i(r))g(x_{i+1}(r))}^2}{d_i(r)}\le C_{T,R},
	\end{equation}
	and
	\begin{equation}\label{eq:Z-propagation}
		\E\abs{Z^{\ell,r,h}(t)}^2
		\le C_{T,R}
		\E\sum_{i=1}^K
		\frac{\abs{\delta_{i\ell}\rho_{r+h}\sig_i^n(X^{n,R}_{r+h})-Y_i^{\ell}(r+h;r)}^2}{d_i(r+h)}.
	\end{equation}
\end{corollary}

Note that MC2(i) follows from \eqref{eq0719a}. To prove MC2(ii), we need to estimate the right-hand side of \eqref{eq:Z-propagation}.

\subsection{The rowwise structural estimate}\label{sub1}

Now, we fix $s$  and suppress the time argument, for example, write $x_i=x_i(s)$. For $u\in\R$ and $v\in\R^K$ with $v_i=u$, define the following structural quotient drift, 
\begin{align}
	\cS_i(u;v)
	&:=\chi_i g(x_{i+1})
	\left(g'(x_i)-\frac{g(x_i)}{d_i}\right)u
	+\chi_i g(x_i)g'(x_{i+1})v_{i+1}\notag\\
	&\quad+g(x_i)g(x_{i+1})
	\ip{\nabla\chi_i(x)}{v}.
	\label{eq:S-row}
\end{align}
The expression $\cS_i$ is exactly
$\ip{\nabla\sigma_i^n(x)}{v}
-\frac{\sigma_i^n(x)}{d_i}u$.
When $u=U_i(s)$ and $v=U(s)$, the part of the drift in
It\^o's formula for $U_i(s)^2/d_i$ that does not involve the
cutoff derivative $R_s(U(s))$ is
\begin{equation}\label{eq:G-row-def}
	\cG_i(u;v)
	:=\rho\left[
	-\frac{b_i(s)}{d_i^2}u^2
	+\frac{2u}{d_i}\sum_{j=1}^K\partial_jb_i(x)v_j
	\right]
	+\frac{\rho^2}{d_i}\abs{\cS_i(u;v)}^2.
\end{equation}

Now,  we define
\begin{equation}\label{eq:Delta-def}
	\Delta_i(s)
	:=
	\rho m_i(s)-\rho^2\gamma_i(x)p(x_{i+1})\alpha_\eta(x_i).
\end{equation}
By \eqref{eq:reserve-def}, \eqref{eq:p-basic}, and \eqref{eq:alpha-bound}, we have
\begin{equation}\label{xi1}
	\begin{aligned}
		\Delta_i(s)
		&=\rho r_i(s)+\rho\gamma_i(x)
		\left[\frac{1}{4}x_{i+1}
		-\rho p(x_{i+1})\alpha_\eta(x_i)\right]\ge0,\\
		\Delta_i(s)&\ge\rho r_i(s),
		\qquad
		\Delta_i(s)\ge\rho(1-\rho)m_i(s).
	\end{aligned}
\end{equation}

\begin{lemma}[Rowwise quotient estimate]\label{lem:rowwise}
	There are constants $c_0>0$ and $C_R<\infty$, independent of $n$, such that
	\begin{equation}\label{eq:rowwise-estimate}
		\cG_i(u;v)
		\le-c_0\frac{\Delta_i(s)}{d_i(s)^2}u^2
		+C_R\left(
		\frac{u^2}{d_i(s)}+
		\sum_{j\ne i}\frac{v_j^2}{d_j(s)}
		\right).
	\end{equation}
\end{lemma}

\begin{proof}
	
	Recall the definition of $b_i(s)$  from \eqref{eq:drift-decomp}.  We first consider the following term:
	\begin{align*}
		&-\rho\frac{x_i\beta_i}{d_i^2}u^2
		+\frac{2\rho u}{d_i}\sum_j\partial_j(x_i\beta_i)v_j\\
		&=\rho\beta_i\frac{x_i+2\eta}{d_i^2}u^2
		+2\rho\frac{x_i\partial_i\beta_i}{d_i}u^2
		+2\rho\frac{x_i u}{d_i}\sum_{j\ne i}\partial_j\beta_i v_j.
	\end{align*}
	On the localized state space, $\beta_i$ and $  \partial_i\beta_i$ are bounded by a constant depending only on $R$.  Moreover, we have, 
	\[
	|uv_j|\le C_R\left(\frac{u^2}{d_i}+\frac{v_j^2}{d_j}\right).
	\]
	Therefore, 
	\begin{equation}\label{eq:self-drift-bound}
		\rho\left[-\frac{x_i\beta_i}{d_i^2}u^2
		+\frac{2u}{d_i}\sum_j\partial_j(x_i\beta_i)v_j\right]
		\le C_R\left(\frac{u^2}{d_i}
		+\sum_{j\ne i}\frac{v_j^2}{d_j}\right).
	\end{equation}
	
	We next consider the terms containing $ r_i$.  Note that, we have 
	$
	2\rho\frac{|\partial_i r_i|}{d_i}u^2\le C_R\frac{u^2}{d_i},
	$ 
	and for $j\ne i$, inequality
	\eqref{eq:weighted-transverse} implies
	$
	d_j|\partial_jr_i|^2
	\le C_R(r_i+\eta).
	$ 
	Then, 
	for arbitrary $\delta_j>0$, applying the Young's inequality gives
	\begin{align}
		2\rho\frac{|\partial_jr_i||uv_j|}{d_i}
		&\le \delta_j\rho(r_i+\eta)\frac{u^2}{d_i^2}
		+\frac{\rho}{\delta_j}
		\frac{|\partial_jr_i|^2}{r_i+\eta}v_j^2\notag\\
		&\le \delta_j\rho r_i\frac{u^2}{d_i^2}
		+C_R\delta_j\frac{u^2}{d_i}
		+\frac{C_R}{\delta_j}\frac{v_j^2}{d_j}.
		\label{eq:m-transverse}
	\end{align}
	Now, we choose the fixed numbers $\delta_j$, such that
	$\sum_{j\ne i}\delta_j\le1/4$. Then
	\begin{equation}\label{eq:beta-transverse}
		\rho\left[-\frac{r_i}{d_i^2}u^2
		+\frac{2u}{d_i}\sum_j\partial_jr_i v_j\right]
		\le-\frac34\rho r_i\frac{u^2}{d_i^2}
		+C_R\left(\frac{u^2}{d_i}
		+\sum_{j\ne i}\frac{v_j^2}{d_j}\right).
	\end{equation}
	
	Since
	\[
	\partial_j\left(\frac{1}{4}\gamma_i x_{i+1}\right)
	=\frac{1}{4}x_{i+1}\partial_j\gamma_i
	+\frac{1}{4}\gamma_i\delta_{j,i+1}.
	\]
	Then, combining the preceding  estimates yields
	\begin{align}
		&\rho\left[
		-\frac{b_i}{d_i^2}u^2
		+\frac{2u}{d_i}\sum_j\partial_jb_i v_j
		\right]\notag\\
		&\quad\le
		-\frac34\rho r_i\frac{u^2}{d_i^2}
		-\frac{1}{4}\rho\gamma_i x_{i+1}\frac{u^2}{d_i^2}
		+\frac{1}{2}\frac{\rho\gamma_i}{d_i}uv_{i+1}
		+\frac{1}{2}\frac{\rho x_{i+1}}{d_i}
		u\ip{\nabla\gamma_i}{v}\notag\\
		&\qquad
		+C_R\left(\frac{u^2}{d_i}
		+\sum_{j\ne i}\frac{v_j^2}{d_j}\right).
		\label{eq:drift-row-final}
	\end{align}
	
	Next, we  decompose
	\begin{equation}\label{eq:S-H-decomp}
		\cS_i(u;v)=S_i(u)+H_i(v),
	\end{equation}
	where
	\begin{align*}
		S_i(u)&=\chi_i g(x_{i+1})
		\frac{d_i g'(x_i)-g(x_i)}{d_i}u,\\
		H_i(v)&=\chi_i g(x_i)g'(x_{i+1})v_{i+1}
		+g(x_i)g(x_{i+1})\ip{\nabla\chi_i}{v}.
	\end{align*}
	For the first term in $H_i$, by \eqref{eq:gprime-bound} and $p(x_i)\le d_i$,  we 
	have
	\[
	\frac{\gamma_i p(x_i)|g'(x_{i+1})|^2v_{i+1}^2}{d_i}
	\le C_R\frac{v_{i+1}^2}{d_{i+1}}.
	\]
	For the second term, the local boundedness of $\nabla\chi_i$ gives
	\[
	\frac{p(x_i)p(x_{i+1})}{d_i}
	\left|\ip{\nabla\chi_i}{v}\right|^2
	\le C_R|v|^2
	\le C_R\sum_{j=1}^K\frac{v_j^2}{d_j}.
	\]
	Therefore, we have
	\begin{equation}\label{eq:H-square}
		\frac{|H_i(v)|^2}{d_i}
		\le C_R\sum_{j=1}^K\frac{v_j^2}{d_j}.
	\end{equation}
	
	Now, we control the cross term $S_iH_i$. Using $p'=2gg'$ and
	$2\chi_i\nabla\chi_i=\nabla\gamma_i$, a direct expansion gives
	\begin{align*}
		\frac{2\rho^2}{d_i}S_i(u)H_i(v)
		&=\rho^2\gamma_i p'(x_{i+1})
		\left(\frac{p'(x_i)}{2d_i}-\frac{p(x_i)}{d_i^2}\right)uv_{i+1}\\
		&\quad+\rho^2p(x_{i+1})
		\left(\frac{p'(x_i)}{2d_i}-\frac{p(x_i)}{d_i^2}\right)
		u\ip{\nabla\gamma_i}{v}.
	\end{align*}
	Note that
	\begin{align*}
		&	-\frac{1}{4}\rho\gamma_i x_{i+1}\frac{u^2}{d_i^2}
		+\frac{1}{2}\frac{\rho\gamma_i}{d_i}uv_{i+1}
		+\frac{1}{2}\frac{\rho x_{i+1}}{d_i}
		u\ip{\nabla\gamma_i}{v}\\
		&	+\frac{\rho^2}{d_i}|S_i(u)|^2+\frac{2\rho^2}{d_i}S_i(u)H_i(v)\notag\\
		=&-\frac{\Delta_i-\rho r_i}{d_i^2}u^2+\frac{\rho\gamma_i}{d_i}
		\left[\frac{1}{2}+\rho p'(x_{i+1})
		\left(\frac{p'(x_i)}2-\frac{p(x_i)}{d_i}\right)\right]u\left(v_{i+1}+ \frac{x_{i+1}}{\gamma_i} \ip{\nabla\gamma_i}{v}\right)\\
		&\quad+\frac{\rho^2}{d_i}
		\left(\frac{p'(x_i)}2-\frac{p(x_i)}{d_i}\right)
		\left[p(x_{i+1})-x_{i+1}p'(x_{i+1})\right]
		u\ip{\nabla\gamma_i}{v}.
	\end{align*}
	
	Since, 
	\[
	p(x_{i+1})-x_{i+1}p'(x_{i+1})
	=-\frac{\eps x_{i+1}^2}{(x_{i+1}+\eps)^2},
	\] 
	$d_i\ge\eta=\eps/8$,  and 
	$|p'(x_i)/2-p(x_i)/d_i|\le3/2$. Then, we have
	\[
	\frac{\rho^2}{d_i}
	\left(\frac{p'(x_i)}2-\frac{p(x_i)}{d_i}\right)
	\left[p(x_{i+1})-x_{i+1}p'(x_{i+1})\right]  u\ip{\nabla\gamma_i}{v} \leq  C_R\left(\frac{u^2}{d_i}+\sum_j\frac{v_j^2}{d_j}\right).
	\]
	For the remaining terms,  if $d_{i+1}\le4d_i$, we obtain
	\begin{equation}\label{x1}
		\begin{aligned}
			&\left|\frac{\rho\gamma_i}{d_i}
			\left[\frac{1}{2}+\rho p'(x_{i+1})
			\left(\frac{p'(x_i)}2-\frac{p(x_i)}{d_i}\right)\right]
			u\left(v_{i+1}+\frac{x_{i+1}}{\gamma_i}
			\ip{\nabla\gamma_i}{v}\right)\right|\\
			&\qquad\le C_R\left(\frac{u^2}{d_i}+\sum_j\frac{v_j^2}{d_j}\right).
		\end{aligned}
	\end{equation}
	If $d_{i+1}>4d_i$, applying  Lemma
	\ref{lem:affine-far-absorption} with 
	$\bar q=\rho\gamma_i/4$ and $\bar \gamma =\rho^2\gamma_i$, yields
	\begin{equation}\label{x2}
		\begin{aligned}
			&\left|\frac{\rho\gamma_i}{d_i}
			\left[\frac{1}{2}+\rho p'(x_{i+1})
			\left(\frac{p'(x_i)}2-\frac{p(x_i)}{d_i}\right)\right]\right|^2d_{i+1}
			\le C_R\frac{\Delta_i-\rho r_i}{d_i^2}.
		\end{aligned}
	\end{equation}
	Moreover,  
	\begin{equation}\label{x3}
		\frac1{d_{i+1}}
		\left|v_{i+1}+\frac{x_{i+1}}{\gamma_i}
		\ip{\nabla\gamma_i}{v}\right|^2
		\le C_R\sum_j\frac{v_j^2}{d_j}.
	\end{equation}
	Combine  \eqref{x1}-\eqref{x3},  we have 
	\begin{align*}
		&-\frac{\Delta_i-\rho r_i}{d_i^2}u^2
		+\frac{\rho\gamma_i}{d_i}
		\left[\frac{1}{2}+\rho p'(x_{i+1})
		\left(\frac{p'(x_i)}2-\frac{p(x_i)}{d_i}\right)\right]
		u\left(v_{i+1}+\frac{x_{i+1}}{\gamma_i}
		\ip{\nabla\gamma_i}{v}\right)\\
		&\qquad\le
		-\frac14\frac{\Delta_i-\rho r_i}{d_i^2}u^2
		+C_R\left(\frac{u^2}{d_i}+\sum_j\frac{v_j^2}{d_j}\right).
	\end{align*}
	Therefore, we obtain
	\begin{align}
		&-\frac{1}{4}\rho\gamma_i x_{i+1}\frac{u^2}{d_i^2}
		+\frac{1}{2}\frac{\rho\gamma_i}{d_i}uv_{i+1}
		+\frac{1}{2}\frac{\rho x_{i+1}}{d_i}
		u\ip{\nabla\gamma_i}{v}
		+\frac{\rho^2}{d_i}|\cS_i(u;v)|^2\notag\\
		&\qquad\le
		-\frac14\frac{\Delta_i-\rho r_i}{d_i^2}u^2
		+C_R\left(\frac{u^2}{d_i}+\sum_j\frac{v_j^2}{d_j}\right).
		\label{eq:diff-row-final}
	\end{align}
	Finally, combining \eqref{eq:drift-row-final} and
	\eqref{eq:diff-row-final} yields
	\[
	\cG_i(u;v)
	\le-\frac34\rho r_i\frac{u^2}{d_i^2}
	-\frac14(\Delta_i-\rho r_i)\frac{u^2}{d_i^2}
	+C_R\left(\frac{u^2}{d_i}
	+\sum_{j\ne i}\frac{v_j^2}{d_j}\right).
	\]
	This proves \eqref{eq:rowwise-estimate} with
	$c_0=1/4$.
	
\end{proof}

\subsection{The spatial-cutoff perturbation}\label{sub2}

For vectors $u,v,y\in\R^K$ with $v_i=u_i$, define
\begin{align}
	\cC_i(u;v,y)
	&:=2b_i(s)\frac{u_iR_s(y)}{d_i}
	+2\rho_s\sigma_i^n(x)
	\frac{R_s(y)}{d_i}\cS_i(u_i;v)\notag\\
	&\quad+\frac{\abs{\sigma_i^n(x)}^2}{d_i}R_s(y)^2.
	\label{eq:C-cutoff-def}
\end{align}
The terms collected in $ \cC_i(u;v,y)$ arise from differentiating the spatial cutoff $\rho_s$. More precisely, they are the additional terms in the drift coefficient of It\^o's formula for $(U^i_s)^2/d_i$. 

\begin{lemma}[Cutoff perturbation bound]\label{lem:cutoff}
	For every $\vartheta>0$ there is $C_{\vartheta,R}<\infty$, independent of $n$, such that
	\begin{align}
		\sum_{i=1}^K\cC_i(u;v^{(i)},y)
		&\le \vartheta\sum_{i=1}^K
		\frac{\Delta_i(s)}{d_i(s)^2}u_i^2\notag\\
		&\quad+C_{\vartheta,R}\left[
		\cQ_s(u)+\cQ_s(y)+
		\sum_{i=1}^K\sum_{j\ne i}\frac{\abs{v_j^{(i)}}^2}{d_j(s)}
		\right],
		\label{eq:cutoff-bound}
	\end{align}
	whenever $v_i^{(i)}=u_i$ for every $i$.
\end{lemma}

\begin{proof}
	By \eqref{eq:cutoff-gradient}, we have 
	\begin{equation}\label{eq:R-bound}
		|R_s(y)|^2
		\le C_R\rho_s(1-\rho_s)|y|^2
		\le C_R\rho_s(1-\rho_s)\cQ_s(y).
	\end{equation}

	First,
	\[
	2x_i\beta_i\frac{u_iR_s(y)}{d_i}
	\le C_R|u_iR_s(y)|
	\le C_R\bigl(\cQ_s(u)+\cQ_s(y)\bigr).
	\]
	Using \eqref{eq:R-bound} and
	$\Delta_i\ge\rho_s(1-\rho_s)m_i$, together with Young's inequality, we obtain
	\begin{equation}\label{eq:cutoff-drift-young}
		\begin{aligned}
			2m_i\frac{|u_iR_s(y)|}{d_i}
			\le&\vartheta\Delta_i\frac{u_i^2}{d_i^2}
			+C_\vartheta\frac{m_i^2|R_s(y)|^2}{\Delta_i}\\
			\le &  \vartheta\Delta_i\frac{u_i^2}{d_i^2}+C_{\vartheta,R}\cQ_s(y). 
		\end{aligned}
	\end{equation}

	The final term in \eqref{eq:C-cutoff-def} satisfies
	\begin{align*}
		\frac{|\sigma_i^n|^2}{d_i}R_s(y)^2
		=\gamma_i\frac{p(x_i)}{d_i}p(x_{i+1})R_s(y)^2
		\le C_R\cQ_s(y).
	\end{align*}

	It remains to estimate the middle term. Use the decomposition
	\eqref{eq:S-H-decomp}. By \eqref{eq:H-square},
	\begin{align*}
		2\rho_s|\sigma_i^nR_s(y)|
		\frac{|H_i(v^{(i)})|}{d_i}
		&\le \rho_s^2\frac{|H_i(v^{(i)})|^2}{d_i}
		+\frac{|\sigma_i^n|^2R_s(y)^2}{d_i}\\
		&\le C_R\left(
		\sum_j\frac{|v_j^{(i)}|^2}{d_j}+\cQ_s(y)
		\right).
	\end{align*}
	Note that
	\begin{align*}
		\frac{|\sigma_i^nS_i(u_i)|}{d_i}
		&=\gamma_i p(x_{i+1})
		\frac{|g(x_i)[d_i g'(x_i)-g(x_i)]|}{d_i^2}|u_i|
		\le C_Rx_{i+1}\frac{|u_i|}{d_i}.
	\end{align*}
	Combine  \eqref{eq:R-bound}, \eqref{xi1}, $m_i\ge\underline\gamma x_{i+1}/4$   and the Young's inequality yields
	\begin{align*}
		2\rho_s|R_s(y)|
		\frac{|\sigma_i^nS_i(u_i)|}{d_i}
		&\le \vartheta\Delta_i\frac{u_i^2}{d_i^2}
		+C_{\vartheta,R}
		\frac{\rho_s^2|R_s(y)|^2x_{i+1}^2}{\Delta_i}\\
		&\le \vartheta\Delta_i\frac{u_i^2}{d_i^2}
		+C_{\vartheta,R}\cQ_s(y).
	\end{align*}
	Combine above estimates, and summing over $i$ proves
	\eqref{eq:cutoff-bound}.
	
\end{proof}

\subsection{Proof of weighted stability}\label{sub3}

\begin{proof}[Proof of Proposition  \ref{prop:weighted-stability}]
	Applying the It\^o's formula to $U_i(s)^2/d_i(s)$.  Note that 
	\[
	\dd d_i(s)=\rho_sb_i(s)\dd s+\rho_s\sigma_i^n(X_s^{n,R})\dd W_s^i,
	\]
	and $U$ satisfies \eqref{eq:generic-tangent}. Then,  a direct calculation gives
	\begin{equation}\label{eq:Q-Ito}
		\dd\cQ_s(U(s))
		=\left[
		\sum_{i=1}^K\cG_i(U_i(s);U(s))
		+\sum_{i=1}^K\cC_i(U(s);U(s),U(s))
		\right]\dd s+\dd M_s,
	\end{equation}
	where $M$ is a continuous local martingale.  Using  Lemma \ref{lem:rowwise} and choosing $\theta<c_0/2$ as in Lemma \ref{lem:cutoff}, we obtain   
	\begin{equation}\label{eq:Q-differential-ineq}
		\dd\cQ_s(U(s))\le C_R\cQ_s(U(s))\dd s+\dd M_s.
	\end{equation}
	
	To justify taking expectations, for $m,N>0$, define
	\[
	\beta_m=\inf\{s\ge a:\langle M\rangle_s-\langle M\rangle_a\ge m\}\wedge t,
	\qquad
	\tau_N=\inf\{s\ge a:\cQ_s(U(s))\ge N\}\wedge t.
	\]
	Integrate \eqref{eq:Q-differential-ineq} up to $s\wedge\tau_N\wedge\beta_m$ and take expectations.  The stopped martingale has zero expectation, so
	\[
	\E\cQ_{s\wedge\tau_N\wedge\beta_m}(U)
	\le\E\cQ_a(\xi)
	+C_R\int_a^s\E\left[
	\1_{\{u\le\tau_N\wedge\beta_m\}}\cQ_u(U)
	\right]\dd u.
	\]
	Let $m\to\infty$, using Fatou on the left and monotone convergence on the right, and applying the  Gronwall's inequality, yields
	\[
	\E\cQ_{s\wedge\tau_N}(U)
	\le \e^{C_R(s-a)}\E\cQ_a(\xi).
	\]
	Since $\cQ_\cdot(U)$ is continuous and finite on $[a,t]$, one has $\tau_N\uparrow t$ almost surely. Then,   Fatou's lemma proves \eqref{eq:weighted-stability}.  Note that, 
	\[
	\abs{U(t)}^2=\sum_i d_i(t)\frac{U_i(t)^2}{d_i(t)}
	\le C_R\cQ_t(U(t)).
	\]
	This proves \eqref{eq:unweighted-stability}.
\end{proof}

\section{Some probabilistic estimates}\label{sec5}

This section develops the probabilistic estimates that translate the boundary behavior of the localized approximations into the short-time control needed for the Malliavin analysis. More precisely,  these estimates will be used to control  the cumulative singular activity over short intervals. 
\subsection{Boundary occupation estimate}\label{subsec:occupation}
The following lemma quantifies the catalyst-weighted time spent near a coordinate face. 
\begin{lemma}[Catalyst-weighted boundary occupation]\label{lem:occupation}
	For each $i=1,\ldots,K$, and  $a>0$, there exists a constant $C_{T,R}>0$, independent of both $n$ and $a$, such that
	\begin{equation}\label{eq:occupation}
		\E\int_0^T\rho_s p(x_{i+1}(s))
		\1_{\{x_i(s)\le a\}}\dd s
		\le C_{T,R}\sqrt a.
	\end{equation}
\end{lemma}

\begin{proof}
	Let $\varphi_a(x)=\sqrt{x+a}$.  Applying  It\^o's formula gives
	\begin{align}
		\dd\varphi_a(x_i(s))
		&=\left[
		\frac{\rho_sb_i(s)}{2\sqrt{x_i(s)+a}}
		-\frac{\rho_s^2\gamma_i(s)p(x_i(s))p(x_{i+1}(s))}
		{8(x_i(s)+a)^{3/2}}
		\right]\dd s+\dd M_s,
		\label{eq:phi-a-Ito}
	\end{align}
	where
	\[
	M_t=\int_0^t
	\frac{\rho_s\sigma_i^n(X_s^{n,R})}{2\sqrt{x_i(s)+a}}\dd W_s^i.
	\]
	It is clear that $M$ is a square-integrable martingale.
	
	Note that $\frac{\rho_sx_i \beta_i}{2\sqrt{x_i(s)+a}}$ is bounded below by $-C_R$. Moreover,  by \eqref{eq:reserve-def}, $m_i\ge\gamma_i x_{i+1}/4\ge\gamma_i p(x_{i+1})/4$. Then, we have
	\begin{align*}
		&\frac{\rho_sm_i}{2\sqrt{x_i+a}}
		-\frac{\rho_s^2\gamma_i p(x_i)p(x_{i+1})}
		{8(x_i+a)^{3/2}}\\
		&\quad\ge
		\frac{\rho_s\gamma_i p(x_{i+1})}{8(x_i+a)^{3/2}}
		\bigl[(x_i+a)-\rho_sp(x_i)\bigr]\\
		&\quad\ge
		\frac{\underline\gamma}{8}\rho_sp(x_{i+1})
		\frac{a}{(x_i+a)^{3/2}}.
	\end{align*}
	Consequently,
	\[
	\dd\varphi_a(x_i(s))
	\ge\left[-C_R+\frac{\underline\gamma}{8}\rho_sp(x_{i+1}(s))
	\frac{a}{(x_i(s)+a)^{3/2}}\right]\dd s+\dd M_s.
	\]
	Integrating over $[0,T]$ and taking expectations gives
	\begin{align*}
		\frac{\underline\gamma}{8}\E\int_0^T\rho_sp(x_{i+1}(s))
		\frac{a}{(x_i(s)+a)^{3/2}}\dd s
		&\le \E\bigl[\varphi_a(x_i(T))-\varphi_a(x_i(0))\bigr]+C_RT.
	\end{align*}
	Note that
	\[
	\begin{aligned}
		\E\phi_a(x_i(T))-\phi_a(x_i(0))
		&\le
		\E\bigl[\sqrt{x_i(T)+a}-\sqrt a\bigr]
		\le C_{T,R}.
	\end{aligned}
	\]
	Consequently,
	\[
	\E\int_0^T
	\rho_s p(x_{i+1}(s))
	\frac{a}{(x_i(s)+a)^{3/2}}\,\dd s
	\le C_{T,R}.
	\]
	
	Finally, on $\{x_i(s)\le a\}$,
	$
	x_i(s)+a\le2a,
	$
	and hence
	\[
	\frac{a}{(x_i(s)+a)^{3/2}}
	\ge
	\frac{1}{2^{3/2}\sqrt a}
	\mathbf 1_{\{x_i(s)\le a\}}.
	\]
	Therefore
	\[
	\E\int_0^T
	\rho_s p(x_{i+1}(s))
	\mathbf 1_{\{x_i(s)\le a\}}\,\dd s
	\le C_{T,R}\sqrt a.
	\]
	
\end{proof}

We further have the following lemma. 
\begin{lemma}\label{cor:dyadic}
	For every $0<h\le1$,
	\begin{equation}\label{eq:dyadic}
		\E\int_0^T\rho_s^2p(x_{i+1}(s))
		\left(1\wedge\frac{h}{x_i(s)}\right)\dd s
		\le C_{T,R}h^{1/2}.
	\end{equation}
\end{lemma}

\begin{proof}
	Use $\rho_s^2\le\rho_s$.  On the set $\{x_i\le h\}$, we apply Lemma  \ref{lem:occupation} with $a=h$.  On the shell
	$2^kh<x_i\le2^{k+1}h$, one has
	$1\wedge h/x_i\le2^{-k}$.  Therefore
	\begin{align*}
		&\E\int_0^T\rho_s^2p(x_{i+1})
		\left(1\wedge\frac{h}{x_i}\right)\dd s\\
		&\quad\le C_{T,R}\sqrt h+
		\sum_{k=0}^\infty2^{-k}
		\E\int_0^T\rho_sp(x_{i+1})
		\1_{\{x_i\le2^{k+1}h\}}\dd s\\
		&\quad\le C_{T,R}\sqrt h
		\left(1+\sum_{k=0}^\infty2^{-k/2}\right) \le  C_{T,R}h^{1/2},
	\end{align*}
	which proves the lemma.
\end{proof}

\subsection{A square-root downcrossing estimate}\label{subsec:downcrossing}

For $r\ge0$, define
\begin{equation}\label{eq:downcross-time}
	\tau_{i,r}^-:=\inf\left\{s\ge r:x_i(s)\le\frac12x_i(r)\right\}.
\end{equation}

\begin{lemma}[Short-time downcrossing]\label{lem:downcross}
	For every $0<h\le1$,
	\begin{equation}\label{eq:downcross}
		\Pp\left(\tau_{i,r}^-\le r+h\mid\F_r\right)
		\le\left(1\wedge C_R \frac{ h}{x_i(r)}\right).
	\end{equation}
\end{lemma}

\begin{proof}
	Applying  It\^o's formula to $R_i(s)=\sqrt{x_i(s)}$.  We obtain
	\begin{equation}\label{eq:sqrt-X-equation}
		\dd R_i(s)=\frac12\rho_s\beta_i(s)R_i(s)\dd s
		+\dd K_i(s)+v_i(s)\dd W_s^i,
	\end{equation}
	where
	\begin{align}
		\dd K_i(s)
		&=\frac{\rho_s}{2R_i(s)}
		\left[m_i(s)-\frac{\rho_s\gamma_i(s)}{4}
		\frac{p(x_i(s))}{x_i(s)}p(x_{i+1}(s))\right]\dd s,
		\label{eq:K-increasing}\\
		v_i(s)&=\frac{\rho_s\chi_i(s)}{2}
		\sqrt{\frac{p(x_i(s))}{x_i(s)}}g(x_{i+1}(s)).
		\label{eq:v-downcross}
	\end{align}
	The bracket in \eqref{eq:K-increasing} is nonnegative. Thus,  $K_i$ is nondecreasing. 
	
	Define the integrating factor
	\[
	E_{r,s}:=\exp\left
	\{-\frac12\int_r^s\rho_u\beta_i(X_u^{n,R})\dd u\right\}.
	\]
	Then
	\begin{equation}\label{eq:downcross-integrating-factor}
		E_{r,s}R_i(s)=R_i(r)+\int_r^sE_{r,u}\dd K_i(u)+M_{r,s},
		\quad
		M_{r,s}=\int_r^sE_{r,u}v_i(u)\dd W_u^i.
	\end{equation}
	Choose $h_0(R)\in(0,1]$ sufficiently small that
	$\e^{C_Rh_0(R)}/\sqrt2<1$, let 
	$c_R=1-\e^{C_Rh_0(R)}/\sqrt2>0$.
	For $0<h\le h_0(R)$,   on the event $\{\tau_{i,r}^-\le r+h\}$,  we have 
	\[
	E_{r,\tau_{i,r}^-}R_i(\tau)
	\le\frac{\e^{C_Rh}}{\sqrt2}R_i(r)
	\le(1-c_R)R_i(r).
	\]
	Since the finite-variation integral in \eqref{eq:downcross-integrating-factor} is nonnegative, we obtain 
	$M_{r,\tau_{i,r}^-}\le-c_RR_i(r)$.  Hence
	\[
	\{\tau_{i,r}^-\le r+h\}
	\subseteq\left\{\sup_{r\le s\le r+h}\abs{M_{r,s}}
	\ge c_RR_i(r)\right\}.
	\]
	Since $R_i(r)=\sqrt{x_i(r)}$ is
	$\mathcal F_r$-measurable, conditional Markov inequality,
	conditional Doob inequality, and conditional It\^o isometry yield
	\begin{align*}
		\Pp(\tau_{i,r}^-\le r+h\mid\F_r)
		&\le\frac{4}{c_R^2x_i(r)}
		\E\left[\int_r^{r+h}E_{r,s}^2\abs{v_i(s)}^2\dd s
		\Bigm|\F_r\right]
		\le C_R\frac{h}{x_i(r)}.
	\end{align*}
	This proves \eqref{eq:downcross} for $0<h\le h_0(R)$.  If $h_0(R)<h\le1$, the localized state space gives a deterministic upper bound $x_i(r)\le B_R$.  Thus,
	$1\wedge h/x_i(r)\ge1\wedge h_0(R)/B_R>0$, and the trivial probability bound is again controlled by the right side of \eqref{eq:downcross} after increasing $C_R$.
\end{proof}

\begin{lemma}[Saturated singular activity]\label{lem:saturated}
	Let $\zeta_i(s)$ be a non-negative progressively measurable process satisfying
	\begin{equation}\label{eq:zeta-upper}
		0\le\zeta_i(s)\le\frac{C_R}{x_i(s)}.
	\end{equation}
	Then, for $0<h\le1$,
	\begin{equation}\label{eq:saturated-conditional}
		\E\left[
		\left(1-\exp\left\{-\int_r^{r+h}\zeta_i(s)\dd s\right\}\right)^2
		\Bigm|\F_r\right]
		\le C_R\left(1\wedge\frac{h}{x_i(r)}\right).
	\end{equation}
	Consequently, for $0<h<t\le T$,
	\begin{align}
		&\E\int_0^{t-h}\rho_r^2p(x_{i+1}(r))
		\left(1-\exp\left\{-\int_r^{r+h}\zeta_i(s)\dd s\right\}\right)^2\dd r
		\le C_{T,R}h^{1/2}.
		\label{eq:saturated-integrated}
	\end{align}
\end{lemma}

\begin{proof}
	Let $A_{r,h}=\{\tau_{i,r}^-\le r+h\}$ and
	$Z_{r,h}=\int_r^{r+h}\zeta_i(s)\dd s$.  On $A_{r,h}^c$, $x_i(s)>x_i(r)/2$, and therefore
	$Z_{r,h}\le C_Rh/x_i(r)$.  Since $(1-\e^{-z})^2\le1\wedge z^2$ for $z\ge0$, it follows that 
	\begin{equation}\label{92}
		(1-\e^{-Z_{r,h}})^2
		\le\1_{A_{r,h}}+C_R\left(1\wedge\frac{h^2}{x_i(r)^2}\right).
	\end{equation}
	Take conditional expectation to \eqref{92}, use
	$1\wedge z^2\le1\wedge z$ for $z\ge0$,  and apply Lemma  \ref{lem:downcross}.   Then proves \eqref{eq:saturated-conditional}.
	
	The weight $\rho_r^2p(x_{i+1}(r))$ is nonnegative and $\F_r$-measurable.  The tower property and \eqref{eq:saturated-conditional} imply
	\begin{align*}
		&\E\left[\rho_r^2p(x_{i+1}(r))(1-\e^{-Z_{r,h}})^2\right]\\
		&\qquad\le C_R\E\left[\rho_r^2p(x_{i+1}(r))
		\left(1\wedge\frac{h}{x_i(r)}\right)\right].
	\end{align*}
	Integrate over $r \in [0, t-h]$ and apply Lemma  \ref{cor:dyadic} to obtain \eqref{eq:saturated-integrated}.
\end{proof}

\subsection{Exponential estimates for semimartingales }\label{subsec:bounded-semi}

The state dependence of $\gamma_i$ introduces a bounded semimartingale correction into the exact tangent-flow ratio \eqref{eq:ratio-formula}.  The next elementary estimate controls it.
\begin{lemma}\label{lem:bounded-semimartingale}
	Fix $T>0$ and let
	\begin{equation}\label{eq:B-general}
		B_{r,s}=\int_r^s a_u\dd u+\sum_{j=1}^K\int_r^s c_u^j\dd W_u^j,
		\qquad 0\le r\le s\le T,
	\end{equation}
	where $a,c^j$ are progressively measurable and
	$
	\abs{a_u}+\sum_{j=1}^K\abs{c_u^j}^2\le C_R.
	$ 
	Then, for $0<h\le1$ with $r+h\le T$,
	\begin{equation}\label{eq:B-exp-increment}
		\E\left[\abs{1-\e^{-B_{r,r+h}}}^2\mid\F_r\right]
		\le C_Rh,
	\end{equation}
	and for every $q\in\R$,
	\begin{equation}\label{eq:B-exp-moment}
		\E\left[\e^{qB_{r,s}}\mid\F_r\right]\le C_{q,T,R}.
	\end{equation}
\end{lemma}

\begin{proof}
	
	Write $B=A+M$.  For $0\le r\le s\le T$,
	$\abs{A_{r,s}}\le C_R(s-r)$ and
	$\langle M\rangle_{r,s}\le C_R(s-r)$, together with the exponential-martingale estimate proves
	\eqref{eq:B-exp-moment}. 
	
	Let $s=r+h$, where $0<h\le1$.  By the conditional BDG inequality and the preceding deterministic bounds, we have
	$
	\E\left[\abs{B_{r,r+h}}^4\mid\F_r\right]\le C_Rh^2.
	$ 
	Combining 
	$
	\abs{1-\e^{-z}}^2\le\e^{2\abs z}\abs z^2
	$ 
	with the conditional Cauchy--Schwarz inequality yields
	\begin{align*}
		\E[\abs{1-\e^{-B}}^2\mid\F_r]
		&\le
		\left(\E[\e^{4\abs B}\mid\F_r]\right)^{1/2}
		\left(\E[\abs{B}^4\mid\F_r]\right)^{1/2}\\
		&\le C_Rh,
	\end{align*}
	where $\e^{4\abs B}\le\e^{4B}+\e^{-4B}$ and
	\eqref{eq:B-exp-moment} is used.  This proves
	\eqref{eq:B-exp-increment}.
\end{proof}

\section{Verification of MC2}\label{sec6}

The proof of fractional regularity compares two Malliavin derivatives whose parameter times differ by a short interval.  The central point is to factor the diagonal part of the tangent flow exactly. In this section, we define, 
$
\Phi_i^n(x):=\chi_i(x)g_n(x_i).
$ 

\subsection{Diagonal tangent-flow factorization}\label{subsec:factorization}

Fix $\ell$ and $r$.  Let $J_{r,s}^\ell$ solve
\begin{equation}\label{eq:J-equation}
	\begin{cases}
		\dd J_{r,s}^\ell=
		\rho_s\partial_\ell b_\ell(X_s^{n,R})J_{r,s}^\ell\dd s
		+\rho_s\partial_\ell\sigma_\ell^n(X_s^{n,R})J_{r,s}^\ell\dd W_s^\ell,\\
		J_{r,r}^\ell=1.
	\end{cases}
\end{equation}
Define the diagonal reference tangent vector
\begin{equation}\label{eq:Yhat-def}
	\widehat Y_i^\ell(s;r)
	:=\delta_{i\ell}\rho_r\sigma_\ell^n(X_r^{n,R})J_{r,s}^\ell.
\end{equation}
Then $\widehat Y^\ell(r;r)=Y^\ell(r;r)$.

Define, 
\begin{equation}\label{eq:lambda-def}
	\lambda_\ell^n(s)
	:=\rho_s m_\ell(s)\frac{g'(x_\ell(s))}{g(x_\ell(s))}
	+\frac12\rho_s^2\gamma_\ell(X_s^{n,R})p(x_{\ell+1}(s))p(x_\ell(s))
	\frac{g''(x_\ell(s))}{g(x_\ell(s))}.
\end{equation}

\begin{lemma}[Positivity and size of the singular activity]\label{lem:lambda}
	For every $s\in[0,T]$,
	\begin{equation}\label{eq:lambda-bounds}
		0\le\lambda_\ell^n(s)\le\frac{C_R}{x_\ell(s)}.
	\end{equation}
\end{lemma}

\begin{proof}
	Note that
	$$
	m_\ell\ge \frac{\gamma_\ell x_{\ell+1}}{4}
	\ge \frac{\gamma_\ell p(x_{\ell+1})}{4},
	\qquad g''\le 0.
	$$
	Using these inequalities together with \eqref{eq:g-convex-combination}, 
	we  have 
	\begin{align*}
		&m_\ell g'(x_\ell)
		+\frac12\rho_s\gamma_\ell p(x_{\ell+1})p(x_\ell)g''(x_\ell)\\
		&\quad\ge\frac14\gamma_\ell p(x_{\ell+1})
		\left[g'(x_\ell)+2\rho_sp(x_\ell)g''(x_\ell)\right]\\
		&\quad\ge\frac14\gamma_\ell p(x_{\ell+1})
		\left[g'(x_\ell)+2p(x_\ell)g''(x_\ell)\right]\ge0,
	\end{align*}
	This proves the lower bound.
	For the upper bound, by \eqref{eq:log-derivative} and localization, we have 
	\[
	\lambda_\ell^n(s)\le\rho_sm_\ell(s)\frac{g'(x_\ell(s))}{g(x_\ell(s))}
	\le\frac{C_R}{x_\ell(s)}.
	\]
\end{proof}

Let
$
h_\ell(x):=\log\chi_\ell(x),
$ 
define 
\begin{align}
	\mathfrak R_\ell^n(s)
	&:=\rho_s\partial_\ell b_\ell(X_s^{n,R})
	-\rho_s x_\ell(s)\beta_\ell(X_s^{n,R})
	\frac{g'(x_\ell(s))}{g(x_\ell(s))}\notag\\
	&\quad-\rho_s\sum_{j=1}^K b_j(X_s^{n,R})\partial_jh_\ell(X_s^{n,R})
	-\frac12\rho_s^2\sum_{j=1}^K\abs{\sigma_j^n(X_s^{n,R})}^2
	\partial_{jj}h_\ell(X_s^{n,R})\notag\\
	&\quad-\frac12\rho_s^2\abs{\sigma_\ell^n(X_s^{n,R})}^2
	\abs{\partial_\ell h_\ell(X_s^{n,R})}^2\notag\\
	&\quad-\rho_s^2\abs{\sigma_\ell^n(X_s^{n,R})}^2
	\partial_\ell h_\ell(X_s^{n,R})
	\frac{g'(x_\ell(s))}{g(x_\ell(s))},
	\label{eq:d-bounded-def}
\end{align}
and, for $j\ne\ell$,
\begin{equation}\label{eq:c-bounded-def}
	c_{\ell j}^n(s):=\rho_s\sigma_j^n(X_s^{n,R})
	\partial_jh_\ell(X_s^{n,R}).
\end{equation}
By \eqref{eq:log-derivative},  we have
\begin{equation}\label{eq:d-c-bounded}
	\abs{\mathfrak R_\ell^n(s)}+
	\sum_{j\ne\ell}\abs{c_{\ell j}^n(s)}^2\le C_R.
\end{equation}
For $r\le s$, define
\begin{align}
	A_{r,s}^\ell&:=\int_r^s\lambda_\ell^n(u)\dd u,
	\label{eq:A-rs}\\
	B_{r,s}^\ell&:=-\int_r^s\mathfrak R_\ell^n(u)\dd u
	+\sum_{j\ne\ell}\int_r^sc_{\ell j}^n(u)\dd W_u^j.
	\label{eq:B-rs}
\end{align}

\begin{lemma}[Semimartingale ratio formula]\label{lem:ratio}
	For $0\le r\le s\le T$,
	\begin{equation}\label{eq:ratio-formula}
		\frac{\Phi_\ell^n(X_r^{n,R})J_{r,s}^\ell}
		{\Phi_\ell^n(X_s^{n,R})}
		=\exp\{-A_{r,s}^\ell-B_{r,s}^\ell\}.
	\end{equation}
	Moreover, for every $q\in\R$,
	\begin{equation}\label{eq:B-uniform-exp}
		\E\left[\exp\{qB_{r,s}^\ell\}\mid\F_r\right]
		\le C_{q,T,R},
	\end{equation}
	and
	\begin{equation}\label{eq:Yhat-energy}
		\E\frac{\abs{\widehat Y_\ell^\ell(s;r)}^2}{d_\ell(s)}
		\le C_{T,R}.
	\end{equation}
\end{lemma}

\begin{proof}
	Apply It\^o's formula to $\log J_{r,s}^\ell$ and to
	$\log\Phi_\ell^n(X_s^{n,R})=h_\ell(X_s^{n,R})+\log g(x_\ell(s))$.
	We obtain
	\[
	\dd\left[
	\log J_{r,s}^\ell-
	\log\Phi_\ell^n(X_s^{n,R})
	\right]
	=-\lambda_\ell^n(s)\dd s
	+\mathfrak R_\ell^n(s)\dd s
	-\sum_{j\ne\ell}c_{\ell j}^n(s)\dd W_s^j.
	\]
	Integrating from $r$ to $s$ and using $J_{r,r}^\ell=1$ proves \eqref{eq:ratio-formula}.
	The estimate \eqref{eq:B-uniform-exp} follows from \eqref{eq:d-c-bounded} and Lemma  \ref{lem:bounded-semimartingale}. Note that  $A_{r,s}^\ell\ge0$ and
	$\abs{\Phi_\ell^n(x)}^2/d_\ell
	=\gamma_\ell(x)p(x_\ell)/d_\ell\le\overline\gamma$, we have 
	\begin{align*}
		\frac{\abs{\widehat Y_\ell^\ell(s;r)}^2}{d_\ell(s)}
		&=\rho_r^2p(x_{\ell+1}(r))
		\frac{\abs{\Phi_\ell^n(X_s^{n,R})}^2}{d_\ell(s)}
		\e^{-2A_{r,s}^\ell-2B_{r,s}^\ell}\\
		&\le C_R\rho_r^2p(x_{\ell+1}(r))\e^{-2B_{r,s}^\ell},
	\end{align*}
	Taking conditional expectation with respect to $\F_r$ and applying \eqref{eq:B-uniform-exp} with $q=-2$ yields  \eqref{eq:Yhat-energy}. 
\end{proof}

\subsection{The non-diagonal remainder}\label{subsec:remainder}

Define
\begin{equation}\label{eq:V-def}
	V^\ell(s;r):=Y^\ell(s;r)-\widehat Y^\ell(s;r).
\end{equation}
Then $V^\ell(r;r)=0$.  For each row $i$, define a vector
$\widetilde V^{\ell,(i)}(s;r)$ by
\begin{equation}\label{eq:Vtilde}
	\widetilde V_j^{\ell,(i)}(s;r)=
	\begin{cases}
		V_i^\ell(s;r),&j=i,\\
		V_j^\ell(s;r)+\delta_{j\ell}\widehat Y_\ell^\ell(s;r),&j\ne i.
	\end{cases}
\end{equation}
Subtracting \eqref{eq:J-equation} from the full tangent equation gives
\begin{align}
	\dd V_i^\ell(s;r)
	&=\left[
	\rho_s\sum_j\partial_jb_i(X_s^{n,R})
	\widetilde V_j^{\ell,(i)}(s;r)
	+b_i(s)R_s(Y^\ell(s;r))
	\right]\dd s\notag\\
	&\quad+\left[
	\rho_s\sum_j\partial_j\sigma_i^n(X_s^{n,R})
	\widetilde V_j^{\ell,(i)}(s;r)
	+\sigma_i^n(X_s^{n,R})R_s(Y^\ell(s;r))
	\right]\dd W_s^i.
	\label{eq:V-equation}
\end{align}

\begin{lemma}[Short-time remainder estimate]\label{lem:V}
	For $r\le s\le r+1$,
	\begin{equation}\label{eq:V-pointwise}
		\E\cQ_s(V^\ell(s;r))\le C_{T,R}(s-r).
	\end{equation}
	Consequently, for $0<h<t\le T$,
	\begin{equation}\label{eq:V-integrated}
		\E\int_0^{t-h}\cQ_{r+h}(V^\ell(r+h;r))\dd r
		\le C_{T,R}h.
	\end{equation}
\end{lemma}

\begin{proof}
	Apply It\^o's formula to $\cQ_s(V^\ell)$, its  drift exactly decomposes into the rowwise structural terms
	and the cutoff perturbation terms. In row $i$, use Lemma \ref{lem:rowwise} with $u=V_i^\ell$ and $v=\widetilde V^{\ell,(i)}$.  Since
	\[
	\sum_j\frac{\abs{\widetilde V_j^{\ell,(i)}}^2}{d_j}
	\le2\cQ_s(V^\ell)
	+2\frac{\abs{\widehat Y_\ell^\ell(s;r)}^2}{d_\ell(s)}.
	\]
	Then, the sum of the structural terms is bounded by
	\[
	\text{structural drift} \leq -c_0\sum_i\Delta_i\frac{\abs{V_i^\ell}^2}{d_i^2}
	+C_R\cQ_s(V^\ell)
	+C_R\frac{\abs{\widehat Y_\ell^\ell(s;r)}^2}{d_\ell(s)}.
	\]
	For the cutoff terms, apply  Lemma \ref{lem:cutoff} with
	$u=V^\ell$, $v^{(i)}=\widetilde V^{\ell,(i)}$, and
	$y=V^\ell+\widehat Y_\ell^\ell e_\ell$, where $e_\ell$ denotes the $\ell$th standard basis vector of $\mathbb R^K$.  For $\vartheta<c_0/2$, we obtain, 
	\begin{align*}
		\text{cutoff drift}
		\le{}&\frac{c_0}{2}\sum_i\frac{\Delta_i}{d_i^2}\abs{V_i^\ell}^2
		+C_R\cQ_s(V^\ell)
		+C_R\frac{\abs{\wh Y_\ell(s;r)}^2}{d_\ell(s)}.
	\end{align*}
	Hence, after localization of the martingale, we have 
	\begin{align*}
		\frac{\dd}{\dd s}\E\cQ_s(V^\ell)
		\le C_R\E\cQ_s(V^\ell)
		+C_R\E\frac{\abs{\widehat Y_\ell^\ell(s;r)}^2}{d_\ell(s)}.
	\end{align*}
	Note that the last expectation is bounded by \eqref{eq:Yhat-energy}.  Therefore,  Gronwall's inequality proves \eqref{eq:V-pointwise}.  Finally,  integrating over $r\in[0,t-h]$ with $s=r+h$ gives \eqref{eq:V-integrated}.
\end{proof}

\subsection{Weighted estimate for the initial defect}\label{subsec:defect}

For $s\ge r$, define
\begin{equation}\label{eq:Gamma-def}
	\Gamma_i^\ell(s;r)
	:=\delta_{i\ell}\rho_s\sigma_\ell^n(X_s^{n,R})
	-Y_i^\ell(s;r).
\end{equation}
At $s=r$, $\Gamma^\ell(r;r)=0$.  At $s=r+h$,
$
\Gamma^\ell(r+h;r)=Z^{\ell,r,h}(r+h).
$

By \eqref{eq:V-def} and \eqref{eq:ratio-formula}, $ \Gamma^\ell$  can be decomposed as
\begin{equation}\label{eq:Gamma-CPV}
	\Gamma^\ell(s;r)=C^\ell(s;r)+P^\ell(s;r)-V^\ell(s;r),
\end{equation}
where only the $\ell$th components of $C^\ell$ and $P^\ell$ are nonzero and
\begin{align}
	C_\ell^\ell(s;r)
	&:=\Phi_\ell^n(X_s^{n,R})
	\left[\rho_sg(x_{\ell+1}(s))-\rho_rg(x_{\ell+1}(r))\right],
	\label{eq:C-defect}\\
	P_\ell^\ell(s;r)
	&:=\rho_rg(x_{\ell+1}(r))\Phi_\ell^n(X_s^{n,R})
	\left[1-\e^{-A_{r,s}^\ell-B_{r,s}^\ell}\right].
	\label{eq:P-defect}
\end{align}

\begin{lemma}[Weighted initial-defect estimate]\label{lem:defect}
	For every $0<h<t\le T$,
	\begin{equation}\label{eq:defect-estimate}
		\sum_{\ell=1}^K\E\int_0^{t-h}
		\cQ_{r+h}(\Gamma^\ell(r+h;r))\dd r
		\le C_{T,R}h^{1/2}.
	\end{equation}
\end{lemma}

\begin{proof}
	Since  $\abs{\Phi_\ell^n(X_s)}^2/d_\ell(s)\le\overline\gamma$, we have, 
	\begin{equation}\label{eq:C-quotient}
		\frac{\abs{C_\ell^\ell(s;r)}^2}{d_\ell(s)}
		\le C_R\abs{\rho_sg(x_{\ell+1}(s))-\rho_rg(x_{\ell+1}(r))}^2.
	\end{equation}
	By \eqref{eq:g-holder},  the cutoff is Lipschitz and $g^2\le x\le C_R$, we have 
	\begin{align*}
		\abs{\rho_sg(x_{\ell+1}(s))-\rho_rg(x_{\ell+1}(r))}^2
		&\le2\abs{g(x_{\ell+1}(s))-g(x_{\ell+1}(r))}^2\\
		&\quad+C_R\abs{X_s^{n,R}-X_r^{n,R}}^2\\
		&\le2\abs{x_{\ell+1}(s)-x_{\ell+1}(r)}
		+C_R\abs{X_s^{n,R}-X_r^{n,R}}^2.
	\end{align*}
	Using \eqref{eq:time-increments} and Cauchy--Schwarz,
	\begin{equation}\label{eq:C-defect-estimate}
		\E\int_0^{t-h}
		\frac{\abs{C_\ell^\ell(r+h;r)}^2}{d_\ell(r+h)}\dd r
		\le C_{T,R}h^{1/2}.
	\end{equation}
	
	For $P^\ell$, we have
	\begin{equation}\label{eq:P-quotient}
		\begin{aligned}
			\frac{\abs{P_\ell^\ell(s;r)}^2}{d_\ell(s)}
			\le&  C_R\rho_r^2p(x_{\ell+1}(r))
			\abs{1-\e^{-A_{r,s}^\ell-B_{r,s}^\ell}}^2\\
			\leq &   C_R\rho_r^2p(x_{\ell+1}(r))  \left((1-\e^{-A_{r,s}^\ell})^2+\abs{1-\e^{-B_{r,s}^\ell}}^2 \right) 
		\end{aligned}
	\end{equation}
	By Lemma \ref{lem:lambda} and Lemma  \ref{lem:saturated},  we have  
	\begin{equation}\label{eq:A-defect-estimate}
		\E\int_0^{t-h}\rho_r^2p(x_{\ell+1}(r))
		(1-\e^{-A_{r,r+h}^\ell})^2\dd r
		\le C_{T,R}h^{1/2}.
	\end{equation}
	Note that, the coefficients of $B_{r,s}^\ell$ satisfy \eqref{eq:d-c-bounded}.  Hence,  Lemma  \ref{lem:bounded-semimartingale} and the tower property give
	\begin{align}
		&\E\int_0^{t-h}\rho_r^2p(x_{\ell+1}(r))
		\abs{1-\e^{-B_{r,r+h}^\ell}}^2\dd r\notag\\
		&\quad\le C_Rh\E\int_0^{t-h}\rho_r^2p(x_{\ell+1}(r))\dd r
		\le C_{T,R}h.
		\label{eq:B-defect-estimate}
	\end{align}
	Combining \eqref{eq:P-quotient}--\eqref{eq:B-defect-estimate},
	\begin{equation}\label{eq:P-defect-estimate}
		\E\int_0^{t-h}
		\frac{\abs{P_\ell^\ell(r+h;r)}^2}{d_\ell(r+h)}\dd r
		\le C_{T,R}h^{1/2}.
	\end{equation}
	Finally, Lemma \ref{lem:V} gives
	\begin{equation}\label{eq:V-defect-estimate}
		\E\int_0^{t-h}\cQ_{r+h}(V^\ell(r+h;r))\dd r
		\le C_{T,R}h.
	\end{equation}
	Combine
	\eqref{eq:C-defect-estimate}, \eqref{eq:P-defect-estimate}, and \eqref{eq:V-defect-estimate} proves \eqref{eq:defect-estimate}.
\end{proof}

\subsection{Completion  the proof of MC2}\label{subsec:fractional}
Combining  Corollary  \eqref{coll1} and Lemma  \ref{lem:defect} gives the  following  result.
\begin{proposition}[Short-interval estimate]\label{prop:short-interval}
	For every $T,R>0$ and $0<h<t\le T$, there exists a constant $ C_{T,R}$ independent of $n$, such that
	\begin{equation}\label{eq:short-interval}
		\sup_{n\ge1}\sum_{\ell=1}^K
		\E\int_0^{t-h}
		\abs{D_{r+h}^\ell X_t^{n,R}-D_r^\ell X_t^{n,R}}^2\dd r
		\le C_{T,R}h^{1/2}.
	\end{equation}
\end{proposition}

\begin{proposition}[Fractional Malliavin regularity]\label{prop:fractional}
	For every $T,R>0$ and every $\alpha\in(0,1/2)$,  there exists a constant $ C_{T,R}$ such that, for every fixed $t\in[0,T]$,
	\begin{equation}\label{eq:fractional}
		\sup_{n\ge1}\E\int_0^T\int_0^T
		\frac{\abs{D_rX_t^{n,R}-D_{r'}X_t^{n,R}}^2}
		{\abs{r-r'}^{1+\alpha}}\dd r\dd r'
		\le C_{T,R,\alpha}.
	\end{equation}
	Thus MC2 (ii) holds.
\end{proposition}

\begin{proof}
	Fix $t\in[0,T]$. If $t=0$, then the estimate is trivial. Hence we
	assume $t>0$. Note that
	\[
	D_rX^{n,R}_t=0,
	\quad r>t.
	\]
	Therefore, the double integral over $[0,T]^2$ decomposes as
	\begin{align}
		\mathbb E
		\int_0^T\int_0^T
		\frac{
			\left|
			D_rX^{n,R}_t
			-
			D_{r'}X^{n,R}_t
			\right|^2
		}
		{|r-r'|^{1+\alpha}}
		\,\dd r\,\dd r'=&\E\int_0^t\int_0^t
		\frac{\abs{D_rX_t^{n,R}-D_{r'}X_t^{n,R}}^2}
		{\abs{r-r'}^{1+\alpha}}\dd r\dd r'\notag\\
		&\quad+2\E\int_0^t\int_t^T
		\frac{\abs{D_rX_t^{n,R}}^2}{(r'-r)^{1+\alpha}}\dd r'\dd r.
		\label{eq:fractional-decomp}
	\end{align}
	Note that, we have 
	\begin{align*}
		&\E\int_0^t\int_0^t
		\frac{\abs{D_rX_t^{n,R}-D_{r'}X_t^{n,R}}^2}
		{\abs{r-r'}^{1+\alpha}}\dd r\dd r'\\
		&=2\int_0^t h^{-1-\alpha}
		\E\int_0^{t-h}
		\abs{D_{r+h}X_t^{n,R}-D_rX_t^{n,R}}^2\dd r\dd h.
	\end{align*}
	Applying  Proposition \ref{prop:short-interval}, we have 
	\begin{equation}\label{eq:inner-bound}
		\begin{aligned}
			\sup_n \mathbb E
			\int_0^t\int_0^t
			\frac{
				\left|
				D_rX^{n,R}_t
				-
				D_{r'}X^{n,R}_t
				\right|^2
			}
			{|r-r'|^{1+\alpha}}
			\,\dd r\,\dd r'
			&
			\leq
			C_{T,R}
			\int_0^t
			h^{-1-\alpha}h^{1/2}\,\dd h
			\\
			&
			=
			C_{T,R}
			\int_0^t
			h^{-1/2-\alpha}\,\dd h\\
			& \leq  	C_{T,R, \alpha}.
		\end{aligned}
	\end{equation}
	For the cross region, Corollary  \eqref{coll1} gives
	\[
	2\E\int_0^t\int_t^T
	\frac{\abs{D_rX_t^{n,R}}^2}{(r'-r)^{1+\alpha}}\dd r'\dd r
	\le C_{T,R}\int_0^t\int_t^T(r'-r)^{-1-\alpha}\dd r'\dd r \leq  	C_{T,R, \alpha}.
	\]
	Combining the two estimates proves \eqref{eq:fractional}.
\end{proof}

\section{Verification of MC3}\label{sec7}

This section verifies MC3 by showing that $L^2$-convergence of the localized approximations implies convergence of both the drift and diffusion coefficients in the required integral norms. 

\begin{proposition}[Convergence of the coefficients]\label{prop:MC3}
	Let $n_m\to\infty$ and suppose
	\begin{equation}\label{eq:MC3-assumed-conv}
		X^{n_m,R}\longrightarrow X^R
		\quad\text{in }L^2([0,T]\times\Omega;\R^K).
	\end{equation}
	Then, for each  $i=1,\ldots,K$,
	\begin{equation}\label{eq:drift-MC3}
		\int_0^T\E\abs{b_i(X_s^{n_m,R})-b_i(X_s^R)}\dd s
		\longrightarrow0,
	\end{equation}
	\begin{equation}\label{eq:diff-MC3}
		\int_0^T\E\abs{\sigma_i^{n_m}(X_s^{n_m,R})
			-\sigma_i(X_s^R)}^2\dd s
		\longrightarrow0.
	\end{equation}
	Consequently, for every $t\in[0,T]$,
	\begin{equation}\label{eq:stoch-int-MC3}
		\int_0^t\sigma_i^{n_m}(X_s^{n_m,R})\dd W_s^i
		\longrightarrow
		\int_0^t\sigma_i(X_s^R)\dd W_s^i
		\quad\text{in }L^2(\Omega).
	\end{equation}
	Thus MC3 holds.
\end{proposition}

\begin{proof}
	It is easy to verify   \eqref{eq:drift-MC3}. We now prove \eqref{eq:diff-MC3}. Note that,  the coefficients of the localized equation  \eqref{eq:localized-approx} vanish
	outside the ball $\{|z|<R+1\}$. Therefore, 
	there exists a deterministic constant $B_R$ such that
	$
	|X_s^{n_m, R}|\le B_R. 
	$
	After passing to an
	convergent subsequence, the same bound holds for
	$X^R$, namely
	$
	|X^R_s|\le B_R, 
	$
	a.s..
	
	For $z,  x\in[0,B_R]^K$,  we have 
	\begin{align}
		\abs{\sigma_i^n(z)-\sigma_i(z)}^2
		&=\gamma_i(z)
		\abs{g_n(z_i)g_n(z_{i+1})-\sqrt{z_iz_{i+1}}}^{2}\notag\\
		&\le2\overline\gamma g_n(z_i)^2
		\abs{g_n(z_{i+1})-\sqrt{z_{i+1}}}^{2}\notag\\
		&\quad+2\overline\gamma z_{i+1}
		\abs{g_n(z_i)-\sqrt{z_i}}^{2}\notag\\
		&\le C_R\epsn,
		\label{eq:sigma-reg-error}
	\end{align}
	and
	\begin{equation}\label{eq:sigma-state-holder}
		\abs{\sigma_i(z)-\sigma_i(x)}^2
		\le C_R\abs{z-x}.
	\end{equation}

	Decompose
	\begin{align*}
		&\int_0^T\E\abs{\sigma_i^{n_m}(X_s^{n_m,R})
			-\sigma_i(X_s^R)}^2\dd s\\
		&\quad\le2\int_0^T\E\abs{\sigma_i^{n_m}(X_s^{n_m,R})
			-\sigma_i(X_s^{n_m,R})}^2\dd s\\
		&\qquad+2\int_0^T\E\abs{\sigma_i(X_s^{n_m,R})
			-\sigma_i(X_s^R)}^2\dd s.
	\end{align*}
	By \eqref{eq:sigma-reg-error}, the first term is at most $C_{T,R}\eps_{n_m}$ and tends to zero.  By \eqref{eq:sigma-state-holder} and \eqref{eq:MC3-assumed-conv} , the second also tends to zero. This proves \eqref{eq:diff-MC3}.  Finally, \eqref{eq:stoch-int-MC3} follows immediately from It\^o's isometry.
\end{proof}

\subsection{Completion of the strong-solution construction}\label{subsec:completion}

\begin{proof}[Proof of Theorem \ref{thm:main}]
	Fix $T,R>0$.  The uniform supremum and time-increment estimates in Lemma 
	\ref{lem:positivity-increments} verify MC1, with $\theta=1$.
	The weighted tangent estimate Corollary \ref{coll1} verifies MC2(i), and Proposition
	\ref{prop:fractional} verifies MC2(ii), for every $\alpha\in(0,1/2)$.
	Proposition \ref{prop:MC3} verifies MC3.  All constants are independent of the regularization index $n$. Therefore, Theorem \ref{thm:MC-implies} applys.  The nonexplosion estimate and the second-moment bound are established in Steps 5 of the proof of Theorem \ref{thm:MC-implies}. This proves \eqref{eq:global-second-moment} and completes the proof of Theorem \ref{thm:main}.
\end{proof}

\appendix
\section{The far-region cross-term absorption}
For completeness we prove the algebraic estimate used in Lemma \ref{lem:rowwise}. 
\begin{lemma}[Far-region absorption]\label{lem:affine-far-absorption}
	Let
	\[
	g(u)=\frac{u}{\sqrt{u+\varepsilon}},\qquad
	p(u)=\frac{u^2}{u+\varepsilon},\qquad
	0<\eta\le\frac{\varepsilon}{8}.
	\]
	Suppose $0\le\bar\gamma\le4\bar q$ and that $\bar q,\bar\gamma$ range in a
	fixed bounded set. With $d=x+\eta$, define
	\begin{align*}
		A(x,y)&=\bar q\frac{y}{d^2}
		-\bar\gamma p(y)\frac{|dg'(x)-g(x)|^2}{d^3},\\
		H(x,y)&=\frac{2\bar q}{d}
		+\bar\gamma\left[
		\frac{p'(x)p'(y)}{2d}-\frac{p(x)p'(y)}{d^2}
		\right].
	\end{align*}
	Then $A(x,y)\ge0$, and there is a constant $C$, uniform over the stated bounded
	parameter set, such that
	\begin{equation}\label{eq:affine-far-absorption}
		H(x,y)^2(y+\eta)\le CA(x,y)
		\quad\text{whenever }y+\eta>4(x+\eta).
	\end{equation}
\end{lemma}

\begin{proof}
	The assertion is trivial when $\bar\gamma=\bar q =0$  and  $\bar\gamma=0 < \bar q $ . Scale
	$x=\varepsilon u$, $y= \varepsilon v$, and
	$\eta=\varepsilon\delta$, where $0<\delta\le1/8$. Let 
	\[
	\bar p(u)=\frac{u^2}{u+1},\qquad
	\bar p'(u)=1-\frac1{(u+1)^2},
	\]
	and
	\[
	a_\delta(u)=
	\frac{(u^2-\delta u-2\delta)^2}
	{4(u+\delta)(u+1)^3}.
	\]
	The far-region condition becomes $v+\delta>4(u+\delta)$. If $\bar\gamma>0$, write 
	$\lambda=\bar q/\bar\gamma\ge1/4$ and define
	\begin{align*}
		h_\lambda(u,v)&=
		\frac{2\lambda}{u+\delta}
		+\frac{\bar p'(u)\bar p'(v)}{2(u+\delta)}
		-\frac{\bar p(u)\bar p'(v)}{(u+\delta)^2},\\
		d_\lambda(u,v)&=\lambda v-\bar p(v)a_\delta(u).
	\end{align*}
	After scaling, it suffices to prove the following  estimate:
	\begin{equation}\label{eq:affine-scaled-goal}
		h_\lambda(u,v)^2(v+\delta)
		\le C\lambda\frac{d_\lambda(u,v)}{(u+\delta)^2}.
	\end{equation}

	For $\lambda=1/4$, a direct expansion gives
	\begin{align}
		\frac14-a_\delta(u)
		&=\frac{P_\delta(u)}{4(u+\delta)(u+1)^3},\label{eq:affine-Pdelta}\\
		P_\delta(u)
		&=(3+3\delta)u^3+(3+7\delta-\delta^2)u^2\notag\\
		&\quad +(1+3\delta-4\delta^2)u+\delta(1-4\delta),
	\end{align}
	whose coefficients are nonnegative. If $0\le u\le1$, then
	$a_\delta(u)\le1/8$. 
	Consequently,
	\[
	d_{1/4}(u,v)
	=\left(\frac14-a_\delta(u)\right)\bar p(v)
	+\frac14\bigl(v-\bar p(v)\bigr)\ge\frac18v.
	\]
	Moreover, $\abs{h_{1/4}(u,v)}\le 2/(u+\delta)$.  Hence
	\eqref{eq:affine-scaled-goal} holds for $0 \leq u\le1$.
	
	If $u\ge1$, \eqref{eq:affine-Pdelta} implies
	\[
	\frac14-a_\delta(u)\ge\frac{3}{32(u+1)}.
	\]
	The far-region condition implies $v\ge4u$ and $\bar p(v)\ge v/2$, so
	\begin{equation}\label{eq:affine-d-lower-large}
		d_{1/4}(u,v)\ge\frac{3v}{64(u+1)}.
	\end{equation}
	Using $\bar p'(w)-1=-(w+1)^{-2}$ and $v \geq 4u$, then 
	\begin{equation}\label{a7}
		\begin{aligned}
			| h_{1/4}(u,v)| 
			&=\left|\left[\frac1{u+\delta}
			-\frac{\bar p(u)}{(u+\delta)^2}\right]
			+\frac{\bar p'(u)\bar p'(v)-1}{2(u+\delta)}
			-\frac{\bar p(u)(\bar p'(v)-1)}{(u+\delta)^2}\right|\\
			& \leq \frac{C}{u^2}.
		\end{aligned}
	\end{equation}
	Combining \eqref{eq:affine-d-lower-large}  and  \eqref{a7}. 
	This proves \eqref{eq:affine-scaled-goal} for $u\ge1$.
	
	Finally, for $\lambda>1/4$,
	\[
	h_\lambda=h_{1/4}+\frac{2(\lambda-1/4)}{u+\delta},
	\qquad
	d_\lambda=d_{1/4}+(\lambda-1/4)v.
	\]
	Using $(a+b)^2\le2a^2+2b^2$ and $v+\delta\le C v$ proves
	\eqref{eq:affine-scaled-goal} for every $\lambda\ge1/4$.  Scaling back gives
	\eqref{eq:affine-far-absorption}.
\end{proof}

\section*{Funding}
This work is supported by
the National Key R\&D Program of China (2022YFA1006102),
the National Natural Science Foundation of China (12471418, 12595294, 12231002),
and the New Cornerstone Science Foundation (NCI202501). 

\section*{AI Use Disclosure}
The authors collaborated with ChatGPT during the research and writing of this manuscript. The authors independently structured and wrote the introduction, and verified every step of each proof proposed by ChatGPT.  The authors take full responsibility for the accuracy, integrity, and final content of the manuscript.

\end{document}